\documentclass[11pt,english,a4paper]{amsart}
\usepackage{graphicx}
\usepackage{amsmath}
\usepackage{amssymb}
\usepackage{mathrsfs}
\usepackage{graphics}
\usepackage{caption}
\usepackage{subcaption}
\usepackage{mathtools}
\usepackage[round]{natbib}
\usepackage{floatrow}
\usepackage{array}
\usepackage{bbm}
 \usepackage{booktabs}
 \usepackage{float}
 \usepackage[labelfont=bf]{caption}
 \usepackage{varioref}
 \usepackage[margin= 1 in]{geometry}
 
 \usepackage{multirow}
  \usepackage{lineno}
 
\def\vero{\textcolor{black}}
\def\veron{\textcolor{black}}
\def\veroo{\textcolor{black}}

 \numberwithin{equation}{section}
 
 \usepackage{setspace}
 \usepackage[hidelinks]{hyperref}
 \hypersetup{
 	colorlinks   = true, 
 	urlcolor     = blue, 
 	linkcolor    = blue, 
 	citecolor   = blue 
 }
 
 \newtheorem{theorem}{Theorem}[section]

 \newtheorem{prop}[theorem]{Proposition}

 \newtheorem{lemma}[theorem]{Lemma}
 
 \newtheorem{remark}[theorem]{Remark}

 \newcommand{\R}{\ensuremath{\mathbb{R}}}

\usepackage[foot]{amsaddr}

\makeatletter
\renewcommand{\email}[2][]{%
	\ifx\emails\@empty\relax\else{\g@addto@macro\emails{,\space}}\fi%
	\@ifnotempty{#1}{\g@addto@macro\emails{\textrm{(#1)}\space}}%
	\g@addto@macro\emails{#2}%
}
\makeatother

	\title[Model-free selection for the mixing coefficient of a MM process]{\veron{A} Model-free selection criterion for \vero{the} mixing coefficient \vero{of} spatial max-mixture models}

\author{Abu-Awwad Abdul-Fattah}
\author{Maume-Deschamps V\'{e}ronique}
\author{Ribereau Pierre }

\address{Camille Jordan Institute - ICJ, Universit\'{e} Claude Bernard Lyon 1, France}

\address{Address for correspondence: A. Abu-Awwad, Department of Applied Mathematics, School of informatics and mathematics, Camille Jordan Institute - ICJ, Universit\'{e} Claude Bernard Lyon 1, 43 boulevard du 11 novembre 1918, F-69622 Villeurbanne Cedex, France.}
\email{abuawwad@math.univ-lyon1.fr, veronique.maume@univ-lyon1.fr, pierre.ribereau@univ-lyon1.fr}

\usepackage{lipsum}

\begin{document}

	\maketitle

	\begin{abstract}
		One of the main concerns in extreme value theory is to quantify the dependence between joint tails. Using stochastic processes that lack flexibility in the joint tail may lead to severe under- or 
		over-estimation of probabilities associated to simultaneous extreme events. Following recent advances in the literature, a flexible \vero{model called} max-mixture model has been introduced for 
		modeling situations where the extremal dependence structure may vary with \vero{the} distance. In this paper we propose a nonparametric model-free selection criterion for \vero{the} mixing 
		coefficient 
		Our criterion is derived from a madogram, a \vero{notion} classically used in 
		geostatistics to capture spatial structures. \vero{The procedure is based on a} nonlinear least squares \vero{between the theoretical madogram and the empirical one}. \vero{We perform a simulation 
			study and apply our} criterion to daily precipitation over the East of Australia.\\
		
	\end{abstract}


\section{Introduction}

Max-stable stochastic processes arise as a fundamental class of models that are able to describe spatial extreme value phenomena. Max-stable process models for spatial data were first constructed 
using the spectral representation \citep{de1984spectral}. Several subsequent works on the construction of spatial max-stable \vero{processes} have been \vero{developed, see e.g. } 
\citep{smith1990max,schlather2002models,kabluchko2009stationary,davison2012geostatistics}. The inference on \vero{spatial processes} is an 
open field that is still in development. Both parametric and nonparametric inference methods are used in the literature.\\
\vero{For a stationary spatial process $X=\{X(s)\/, s\in \mathcal{S}\}$, the Asymptotic Dependence (AD) is characterized  by $\chi (h)>0$, with:}

\begin{equation}
	\label{Asymptotic dep and indep}
	\chi(h)=\lim_{x\rightarrow\infty}\mathbb{P}\{X(s)>x | X(s+h)>x\}, \qquad s, s+h \in \mathcal{S}\/.
\end{equation}
\vero{This means that, for an AD process,}  a large event at location $s+h$ leads to a non-zero probability of a similarly large event at location $s$ for some spatial lag vector $h$. \vero{On one 
	other hand, a process is Asymptotically Independent (AI) if $\chi(h)=0$ for any $h$. This is achieved e.g. for Gaussian processes, see \citep{sibuya1960bivariate}. }
\\

Within \vero{the} class of \vero{max-stable} models, only two types of dependence structures are \vero{possible: either the process is AD or it is independent.} 
This restriction leads to a drawback of max-stable processes that they are too coarse to describe multivariate joint tails with 
asymptotic independence sufficiently accurately. Particularly, fitting asymptotic dependent models to asymptotically independent data leads over/under estimation of probabilities of extreme joint 
events, since there is a mis-placed assumption that the most extreme marginal events may occur simultaneously \citep{coles1999dependence}. \citep{thibaud2013threshold, davison2013geostatistics} introduced recent examples about practical difficulties to identify whether a data set should be modeled using an asymptotically dependent or asymptotically 
independent.\\

\citep{doi:10.1093/biomet/asr080} introduced a new class of models, so-called max-mixture models, to capture both \vero{AD and AI}. The basic idea is 
to mix max-stable and asymptotic independent processes. \vero{Let} $a \in [0,1]$, then the  max-mixture (\vero{MM}) model \vero{is defined} as 
\begin{equation*}
	Z(s) =\max\{\ a X(s),(1-\ a)Y(s)\},\qquad s\in \mathcal{S}
\end{equation*}

where $X(s)$ \vero{is}  a stationary  max-stable process  and $Y(s)$ \vero{is} a stationary \vero{AI process, so that the parameter $a$ represents the proportion of AD in the process $Z$}.\\ 


%
%
%

In this paper, we are \vero{concerned} with constructing a model-free criterion \vero{to choose a realistic value for the mixing parameter $a$.} 
Our objective is not to model extremal dependence of joint tails but to set up a 
statistical criterion that facilitate the modelling of the spatial data with suitable behaviour. \vero{We shall use least squares on the $F^\lambda$-madogram. In \citep{bel2008assessing}, a 
	madogram based test on the AD of a max-stable process is proposed, while in \citep{abd}, a parametric test on $a$ for max-mixture processes is developed.}
\\

The paper is organized as follows. Section \ref{sec:models} reviews spatial extremes \vero{processes}. The proposed $F^{\lambda}-$madogram for max-mixture \vero{models} and \vero{the selection 
	criterion for the} mixing coefficient $a$  are developed in Section \ref{sec:mado}, while Section \ref{sec:simu} illustrates the performance of our method through a number of simulation studies. We 
conclude with an illustration of spatial analysis of precipitation in Australia in Section \ref{real_data}.


\section{Spatial extremes processes: models}\label{sec:models}
\vero{Throughout our work, $X=\{X(s)\/, s\in \mathcal{S}\}$, $\mathcal{S} \subset \R^d$ (generally, $d=2$) is a spatial process, it will be assumed to be stationary and isotropic.}
\subsection{Max-stable processes}
Suppose that $\{Y_{i}(s): s \in \mathcal{S} \subset \mathbb{R}^{d} \}$, $i=0,1,2,...,$ are i.i.d replicates of a random process $Y(s)$, and that there are sequences of continuous functions $\{a_{n}(s)>0\}$ and $\{b_{n}(s)\}$ such that, the rescaled process of maxima,\\
\begin{equation}
	\bigvee_{i =1 }^{n} \frac{Y_{i}(s) - b_n(s)}{a_n(s)} \buildrel d \over \rightarrow X(s), \quad n\rightarrow \infty
\end{equation}
\\

where the limiting random process $X$ is assumed to be non-degenerate. By \citep{de2006spatial} the class of the limiting processes $X(s)$ coincides with the class of max-stable processes. This definition of MS processes offers a natural choice for modeling spatial extremes.
\\

\vero{The} univariate extreme value theory, \vero{implies that} the marginal distributions of $X(s)$ are Generalized Extreme value (GEV) distributed, and without  loss of generality the margins can 
transformed to a simple MS process called standard Fr\'{e}chet distribution, $\mathbb {P}(X(s)\leq z)= \exp \{-z^{-1}\}$.
\\

Following \citep{de1984spectral, schlather2002models}, a \vero{simple} MS process $X(s)$ \vero{has} the following representation

\begin{equation}
	\label{representation}
	X(s)= \max_{k\geq 1} Q_{k}(s)/P_{k}, \qquad s \in \mathcal{S}.
\end{equation}

where $Q_{k}(s)$  are independent replicates of a non-negative stochastic process $Q(s)$ with unit mean at each $s$, and $P_{k}$ are the points of a unit rate Poisson process $(0,\infty)$.
\\

For $K \in \mathbb{N}\setminus\{0\}$, $s_{1},\ldots,s_{K} \in \mathcal{S}$, and $x_{1},\ldots,x_{K}>0$, the finite $K$-dimensional distribution of the process $X$ owing to the representation( \ref{representation}) is given by
\begin{equation}
	- \log \mathbb{P}(X(s_1)\leq z_{1},\ldots,X(s_K)\leq z_{K})= \mathbb{E}  \left[ \bigvee_{k=1}^{K} \left\{\frac{Q (s_k)}{z_k} \right\} \right]=V_{s_1\/,\ldots\/,s_K}(z_{1},\ldots,z_{K})
\end{equation}
\\

where $V_{s_1\/,\ldots\/,s_K}(.)$ is called the exponent measure. It summarises the structure of extremal dependence, and satisfies the property of homogenity of order $-1$ and 
$V_{s_1\/,\ldots\/,s_K}(\infty,...,z,...,\infty)=z^{-1}$.
\vero{It has to be noted that
	$$-z \log \mathbb {P}\{X(s_1)\leq z,..,X(s_K)\leq z\}=  V_{s_1\/,\ldots\/,s_K}(1,...,1)={\theta_{s_1\/,\ldots\/,s_K}}\/,$$ 
	The coefficient $\theta_{s_1\/,\ldots\/,s_K}$ } is known as the extremal coefficient. \vero{It} can be seen as a summary of 
extremal dependence with two boundary values. The complete independence is achieved when $\theta_{s_1\/,\ldots\/,s_K}=1$, while complete independence is achieved when $\theta_{s_1\/,\ldots\/,s_K}=K$. In the bivariate case, the AI 
and AD between a pair of random variables $Z_{1}$ and $Z_{2}$, with marginal distributions $F_{1}$ and $F_{2}$, \vero{may be identified by}
\begin{equation}\label{eq:chi}
	\chi = \lim_{u \, \to \, 1^-} P(F_{1}(Z_{1} )> u | F_{2}(Z_{2}) > u ))  \/.
\end{equation}
\vero{The cases}  $\chi=0$ and $\chi> 0$ represent AI and AD, respectively, \citep{joe1993parametric}. \vero{This coefficient is related to the pairwise} extremal 
coefficient $\theta$ through the relation $\chi=2-\theta$.\\
Since both dependence functions $\theta$ and $\chi$ are useless \vero{for AI processes}, \citep{coles1999dependence} proposed a new dependence \vero{coefficient which measures the strength of 
	dependence for AI processes:}
\begin{equation}\label{eq:chi_bar}
	\bar{\chi}=\lim_{u \, \to \, 1^-}\bar{\chi}(u)=\lim_{u \, \to \, 1^-}\frac{2\log P(F(Z(s))>u)}{\log P(F(Z(s))>u, F(Z(s+h))>u)}-1
\end{equation}
AD (respectively AI) is achieved if and if $\bar{\chi}=1$ (resp. $\bar{\chi}<1$).
\\
Another dependence model for bivariate joint tails \vero{was} introduced by \citep{ledford1996statistics} 
\begin{equation}\label{eq:tawn}
	\mathbb{P}(X(s)>x, X(s+h)>x)\sim x^{-1/\eta(h)} \mathcal{L}_h(x), \text{for} \quad x\rightarrow \infty
\end{equation}

where $X$ is a stationary spatial process with unit Fr\'{e}chet margins, $\mathcal{L}_h(.)$ \vero{is} a slowly varying function at $\infty$ and the tail dependence coefficient $\eta(h) \in (0,1]$. 
AI \vero{corresponds} to $\eta(h)<1$.\\

Different choices for the process $Q(s)$ in (\ref{representation}) lead to more or less flexible models for spatial maxima. Commonly used models are the Guassian extreme value process 
\citep{smith1990max}, the extremal 
Gaussian process \citep{schlather2002models}, the Brown-Resnick process \citep{kabluchko2009stationary}, and the extremal$-t$  process \citep{opitz2013extremal}. Below, we list these four 
specific examples of max-stable models.
\\


The storm profile model \citep{smith1990max}, \vero{is} defined by taking $Q_{k}(s)= f(s-W_{k})$ in Equation (\ref{representation}), where $f$ is \vero{the density function of a Gaussian random 
	vector} with covariance matrix $\Sigma \in \mathbb{R}^{2\times 2}$. \vero{The function $f$} plays a major role as it determines the shape of the storm events. $W_{k}$ is a homogenous Poisson process. 
The bivariate \vero{exponent function} of the Smith model \vero{has} the form

\begin{eqnarray*}
	\lefteqn{-\log \mathbb{P}[X(s) \leq x_{1}, X(s+h ) \leq x_{2} ] }\\
	&&= \frac{1}{x_{1}}\Phi\left( \frac{\beta(h)}{2}+\frac{1}{\beta(h)}\log\left(\frac{x_{2}}{x_{1}}\right)\right )+ 
	\frac{1}{x_{2}}\Phi\left( \frac{\beta(h)}{2}+\frac{1}{\beta(h)}\log\left(\frac{x_{1}}{x_{2}}\right)\right)
\end{eqnarray*}
where $\beta(h)=\sqrt{h^{T}\Sigma^{-1}h}$  and $\Phi$ is the standard normal distribution \vero{function}. In this case the 
extremal coefficient is equal to $\theta(h)=2\Phi\{\beta(h)/2\}$.
\\


A model originally due to \citep{schlather2002models} \vero{is} the Truncated Extremal Gaussian (TEG) model and has been exemplified \vero{in} \citep{davison2012geostatistics}. This process 
\vero{is obtained by taking} $Q_{k}(s)= c \max (0,\varepsilon_{k}(s))\mathbbm{1}_{\mathcal{A}_{k}}{(s-R_{k})}$, \vero{where} $\varepsilon_{k}(s)$ are independent replicates of a 
stationary Gaussian process 
$\varepsilon=\{\varepsilon(s),s \in \mathcal{S}\}$ with zero mean, unit variance and correlation function $\rho(.)$. $\mathbbm{1}_{\mathcal{A}}$ is the indicator function of a compact random set 
$\mathcal{A} \subset \mathcal{S}$, $\mathcal{A}_{k}$ are \vero{independent} replicates of $\mathcal{A}$ and $R_{k}$ are points of \vero{a} Poisson process with a unit rate on $\mathcal{S}$. The 
constant $c$ is chosen 
to satisfy the constraint $\mathbb{E}\{Q_k(s)\}=1$.\\

The bivariate \vero{exponent function} of \vero{a} TEG model in the stationary case has the form
\begin{equation}
	\label{truncated}
	-\log \mathbb{P}[X(s) \leq x_{1}, X(s+h ) \leq x_{2} ] =\left( \frac{1}{x_1}+\frac{1}{x_2}\right) \left[ 1- \frac{\alpha(h)}{2}\left(1- \sqrt{1- \frac{2 (\rho(h)+1)x_{1} x_{2}} {(x_{1}+x_{2})^{2}}}\right)\right] 
\end{equation}

where  \vero{$\alpha(h)= (1- h/2r) \mathbbm{1}_{[0,2r]}$ if $\mathcal{A}$ is} a disk of fixed radius $r$. This yields $\theta(h)= 2 - \alpha({h}) \{1-\sqrt{(1-\rho(h))/2}\}$.
\\


The max-stable Brown-Resnick (BR) process model proposed by \citep{brown1977extreme, kabluchko2009stationary} is a
stationary max-stable process that can be constructed \vero{with} $Q_{k}(s)= \exp\{\varepsilon_{k}(s)-\gamma(s)\}, s \in \mathcal{S}$, \vero{where} $\varepsilon_{k}(s)$ denotes \vero{a} Gaussian 
process with semivariogram $\gamma(h)$. The bivariate \vero{exponent function of a }  BR \vero{process is:}

\begin{eqnarray*}
	\lefteqn{-\log \mathbb{P}[X(s) \leq x_{1}, X(s+h ) \leq x_{2} ] }\\
	&&= \frac{1}{x_{1}}\Phi\left( 
	\sqrt{\frac{\gamma(h)}{2}}+\frac{1}{\sqrt{2\gamma(h)}}\log\left(\frac{x_{2}}{x_{1}}\right)\right )+ \frac{1}{x_{2}}\Phi\left( 
	\sqrt{\frac{\gamma(h)}{2}}+\frac{1}{\sqrt{2\gamma(h)}}\log\left(\frac{x_{1}}{x_{2}}\right)\right)
\end{eqnarray*}

where \vero{$\gamma$ and $\Phi$ denote respectively} the semivariogram and \vero{the} standard normal distribution function. In particular, when the variogram 
$2\gamma(h)=h^{T}\Sigma^{-1}h$ for some covariance matrix $\Sigma$, \vero{we recover the bivariate distribution function of a Smith model}.The pairwise extremal coefficient for \vero{a} Brown-Resnick 
process \vero{is} $\theta(h)=2\Phi\{\sqrt{\gamma(h)/2}\}$.
\\


The extremal$-t$ max-stable process proposed \vero{in} \citep{opitz2013extremal,ribatet2013extreme} can be constructed by using $Q_{k}(s)= 
\left\{m_{v}^{-1/v} 
T_k(s)\right\}^{v}$, where $T_k$ is \vero{a} zero mean Gaussian process with correlation function $\rho$, $v\ge 1$, $1/m_{v}=\sqrt{\pi}2^{v/2 -1} \Gamma(\frac{v+1}{2})$, $\Gamma(.)$ is the gamma function. This process has 
the \vero{bivariate exponent function:}

\begin{eqnarray*}
	\lefteqn{-\log \mathbb{P}[X(s) \leq x_{1}, X(s+h ) \leq x_{2} ]  }\\
	&&=\frac{1}{x_{1}}T_{v+1}\left( {\alpha}{\rho(h)}+{\alpha}\left(\frac{x_{2}}{x_{1}}\right)^{1/v}\right )+ 
	\frac{1}{x_{2}}T_{v+1}\left( {\alpha}{\rho(h)}+{\alpha}\left(\frac{x_{1}}{x_{2}}\right)^{1/v}\right)
\end{eqnarray*}

where $T_{v}$ is the \vero{distribution function} of a Student random variable with $v$ degrees of freedom and $\alpha=[{v+1}/ \{1-\rho^{2}(h)\}]^{1/2}$. For \vero{an} extremal$-t$ process the 
degrees of freedom $v$ controls the upper bound of the extremal coefficient: $\theta(h)=2T_{v+1}\left(\sqrt{(v+1) \frac{1-\rho(h)}{1+\rho(h)}}\right)$.
\\
%
%

\vero{In this paper, we shall make intensive use of the so-called $F-$madogram that is based on a classical geostatistical tool, the madogram \citep{matheron1987suffit} 
	. It has been introduced in \citep{cooley2006variograms}. Let $X$ is a stationary max-stable random process. The marginal distribution function is denoted 
	by $F$. The $F$-madogram is defined by:} 

\begin{equation}
	\label{F-madogram}
	\nu_{F}(h)=\frac{1}{2}\mathbb{E}[|F(X(s))-F(X(s+h))|], \qquad 0 \leq \nu_{F}(h) \leq 1/6\/.
\end{equation}
\vero{The bounds of the $F$ madogram} correspond \vero{respectively} to complete dependence and independence. Due to the one-to-one relationship between the extremal dependence function and the 
$F-$madogram, a simple estimator for $\theta(h)$ \vero{can be derived:}

\begin{equation}
	\widehat{\theta}(h)=\frac{0.5+\widehat{\nu}_{F}(h)}{0.5-\widehat{\nu}_{F}(h)}
\end{equation}
where $\widehat{\nu}(h) = \frac{1}{2 N} \sum_{i=1}^{N} |\widehat{F}\{x_{i} (s)\} - \widehat{F}\{x_{i} (s+h)\}|$, $x_{i} (s)$ and $x_{i} (s+h)$ are the $i-$th observations of the random field at 
locations $s$ and $s+h$. \vero{$\widehat{F}$ is the empirical distribution function}, i.e, $\widehat{F}(z)= (N+1)^{-1} \sum_{i=1}^{N} \mathbb{I}_{Z^i(s_j)\leq z}$, where $\mathbb{I}(.)$ is the indicator function.
\\

The so-called $F^{\lambda}-$madogram has been introduced for max-stable models by \citep{bel2008assessing} as a generalization of $F-$madogram (\ref{F-madogram}): \vero{for any $\lambda >0$, let}
\begin{equation}
	\label{lambda-madogram}
	\nu_{F^\lambda}(h)= \frac{1}{2}\mathbb{E} \left[|F^{\lambda}{\{X(s)\}} - F^{\lambda}{\{X(s+h)\}}|\right]\/.
\end{equation}

A nonlinear least squares procedure has been proposed by \citep{bel2008assessing} based on $F^{\lambda}-$madogram  to compute an estimator for the extremal dependence function that may outperforms 
other known estimators. \vero{In that work, it has been found by some trials that good estimations are obtained for $\lambda \in [2\/,3]$.}\\


\subsection{Hybrid models of spatial extremal dependence}

Although max-stable models \vero{seem} to be suitable for modeling extremely high threshold exceedances, asymptotic independence models may show a better fit at finite thresholds. Due to difficulty 
or 
impossibility in practice to decide whether a dataset should be modeled using AD or AI, \citep{doi:10.1093/biomet/asr080} have been introduced the hybrid spatial dependence models which are able to capture 
both AD and AI.
\\

\vero{Consider $Y'$ a stationary Gaussian  process. Let} $Y(s) = -1/\log(\varPhi(Y'(s)))$ \vero{then, $Y'$} is \vero{an} AI process with unit Fréchet \vero{marginal distributions. Another} class of 
AI processes \vero{called inverted} max-stable 
processes has been proposed by \citep{doi:10.1093/biomet/asr080}. \vero{They are defined as}

\begin{equation}
	\label{IMS}
	Y(s) = -1/\log\{{1-\exp{[-Y'(s)^{-1}]}}\}
\end{equation}
\vero{where $Y'$ is a simple max-stable process with extremal coefficient $\theta_{Y'}$. We a slight abuse of language, we shall denote $\theta_{Y'}$ by $\theta_Y$.}
With this construction, each max-stable process may be transformed \vero{into an} AI independent counterpart. \vero{This inverted max-stable process (IMS) satisfies (\ref{eq:tawn}) and }
$\eta(h)=1/\theta_{Y}(h)$. The bivariate \vero{distribution function is given by}
\begin{equation*}
	\mathbb{P}(Y(s)\leq y_1, Y(s+h) \leq y_2)= -1 +\exp(-y^{-1}_{1})+\exp(-y^{-1}_{2})+\exp\{-V_{Y}{[\omega(y_{1}),\omega(y_{2})]}\}
\end{equation*}

where $V_{Y}$ is the exponent measure of the bivariate extreme-distribution of $\{{Y'(s),Y'(s+h)}\}$, and $\omega(y)=-1/\log\{{1-\exp{[-y^{-1}]}}\}$. \vero{We a slight abuse of language, we shall say 
	that $V_Y$ is the exponent measure of $Y$. }\\

\vero{We are now in position to define the max-mixture processes that we will be working on.}
Let $X$ be a \vero{simple} max-stable process with bivariate extremal coefficient $\theta_X$, and $Y$ be an inverted max-stable process with coefficient of tail dependence $\eta$. Assume that 
$X$ and $Y$ are independent. Then for a mixture proportion $\ a \in [0,1]$, the spatial max-mixture process proposed by \citep{doi:10.1093/biomet/asr080} \vero{is} defined 
as
\begin{equation}
	\label{max-mixture}
	Z(s) =\max\{\ a X(s),(1-\ a)Y(s)\}.
\end{equation}
Clearly, models that are only AD or AI are submodels of $Z$, obtained \vero{for} $a=1$, $a=0$, respectively. The bivariate joint survivor function of the process $Z$ \vero{satisfies}
\begin{equation}
	\mathbb{P} [Z(s)>z , Z(s+h)>z]=\frac{a (2 - \theta(h))}{z} +\left(\frac{1-a}{z}\right)^{1/\eta(h)} + O(z^{-2})  \qquad\text{as}\quad  z \rightarrow \infty
	\label{survivor function}
\end{equation}
If $h_{0}=\inf{\{h: \theta(h)=2\}}$ is finite, then the process $Z$ is AD up to distance $h_{0}$, and AI for longer distances. The bivariate CDF for a pair of sites $(Z(s),Z(s+h))$ is 
\vero{straightforwardly} obtained by the independence between $X(s)$, $Y(s)$

\begin{equation}
	\label{Biv MM}
	\mathbb{P}(Z(s)\leq z_1, Z(s+h) \leq z_2)=\mathbb{P}\left(X(s)\leq\frac{z_1}{\ a}, X(s+h)\leq \frac{z_2}{\ a}\right) \mathbb{P}\left(Y(s)\leq\frac{z_1}{1-\ a},Y(s+h)\leq\frac{z_2}{1-\ a}\right)
\end{equation}

Thus, in the case where $X(s)$ is a max-stable process and $Y(s)$ is a inverted max-stable process, the distribution function in (\ref{Biv MM}) has the form
\begin{eqnarray}
	\mathbb{P}(Z(s)\leq z_1, Z(s+h) \leq z_2)=\exp\{- a V_{X}(z_1,z_2)\} \times \{-1+ \exp[(a-1)/z_1]\\ + \exp[(a-1)/z_2]+\exp[-V_{Y}[\omega((1-a)/z_1), \omega((1-a)/z_2)]\}\nonumber
	\label{distribution MM}
\end{eqnarray}

where $V_{X}$ and $V_{Y}$ are the bivariate exponent measures for $X$ and $Y$ respectively.
\\


\vero{Figure} \ref{Eta} displays two simulated images of AI processes over the $[0,10]^2$ square. The corresponding functions $\eta(h)$ are also represented. According to \citep{ledford1997modelling}, 
the case 
$\eta(h)=1/2$ corresponds to the near-independence, the AI process constructed from a Brown-Resnick process (this case corresponds to the isotropic Smith process) allows asymptotic independence but tends 
to near-indpendence for long distances. While an AI process constructed from extremal$-t$ process presents a stronger dependence in the asymptotic independence when $h$ is sufficiently large.

\begin{figure} [H]
	\includegraphics[width=0.99\linewidth, height=5cm]{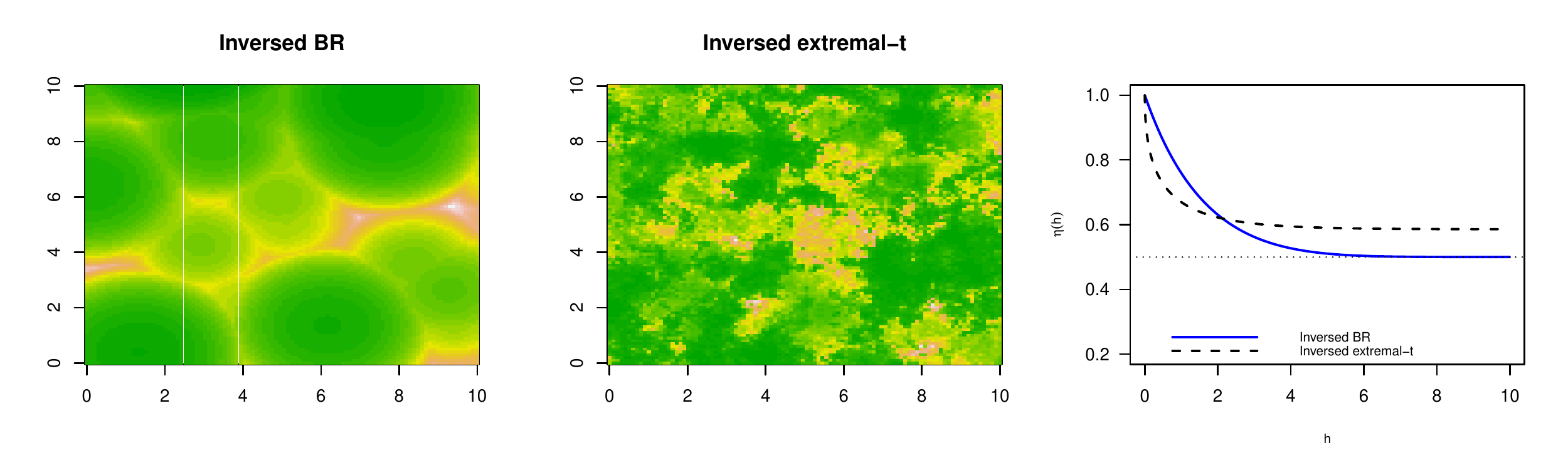}
	
	\caption{Simulations of two inverted max-stable processes (\ref{IMS}) on the logarithm scale. \textbf{Left panel}: simulated image of AI process constructed by inverting an isotropic inverted Brown-Resnick process with variogram $2\gamma(h)= (h / 1.5)^{2}$. \textbf{Middle panel}: simulated image of AI process constructed by inverting an isotropic extremal$-t$ process with $v=1$ degrees of freedom and exponential correlation function $\rho(h)=\exp(- h/1.5)$. On the \textbf{Right panel}: associated functions $\eta(h)$.}
	\label{Eta}	
\end{figure}


\vero{Figure} \ref{Simulationplot} displays five simulated images of the max-mixture model over the $[0,10]^2$ square according to different values of the mixing coefficient $a$. In order to show the 
role of the mixing coefficient, the values in the images are acquired by considering the simulation when $a=1$(max-stable process) and $a=0$ (inverted max-stable process). It is noteworthy that the smoothness decreases \vero{as $a$ increases}. \vero{Figure} \ref{SimulationApp} (\nameref{sec:Appendix.A}) displays further examples of simulated 
images of max-mixture models.

\begin{figure} [H]
	\includegraphics[width=0.99\linewidth, height=5cm]{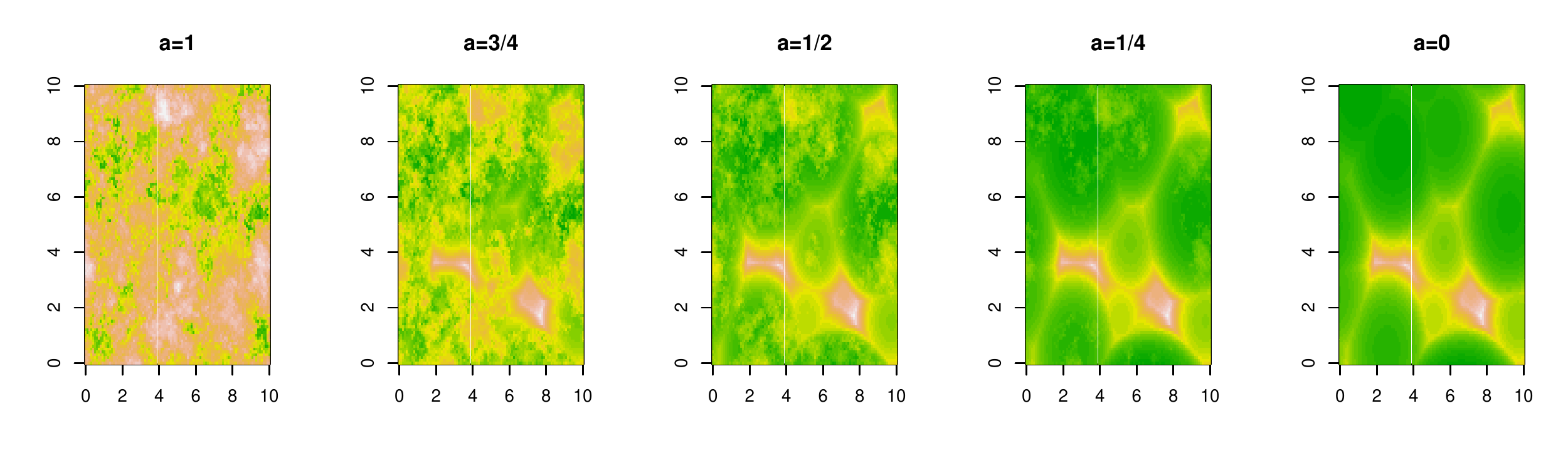}
	
	\caption{Simulations of the max-mixture model (\ref{max-mixture}) on the logarithm scale according different values of mixing coefficient $a \in \{1,0.75,0.5,0.25,0\}$. $X$ is an isotropic extremal$-t$ process with $v=1$ degrees of freedom and exponential correlation function $\rho(h)=\exp(- h)$, and $Y$ is an isotropic inverted Brown-Resnick process with variogram $2\gamma(h)= (h / 1.5)^{2}$. }
	\label{Simulationplot}	
\end{figure}

\vero{Figure} \ref{Theta} illustrates how the spatial extremal dependencies vary regarding to $a$. Based on (\ref{survivor function}), the measure $\chi_{Z}$ associated to the process $Z$ 
can be computed for a distance $h$ as $\chi_Z(h)= a \chi_X(h)$, \vero{see \citep{bacro2016flexible}.} The hybrid model extends traditional dependence modeling within the AD class and is appropriate 
when AD is present at all distances because it permits to capture a second order in the dependence structure which is not possible with a max-stable model.

\begin{figure} [H]
	\includegraphics[width=0.99\linewidth, height=5cm]{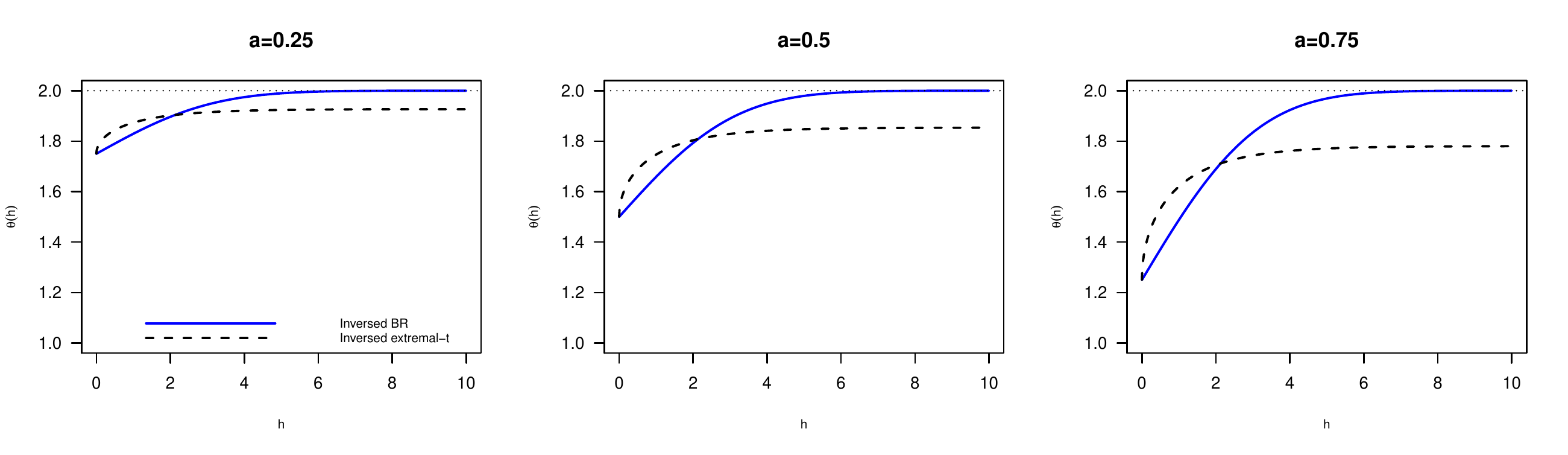}
	
	\caption{Associated $\theta(.)$ functions for a max-mixture process $Z(.)$ based on two max-stable processes $X(.)$ and  $X'(.)$ that belong to the same family with the same parameters used in Figure \ref{Eta} for $a \in \{0.25, 0.5, 0.75\}$.}
	\label{Theta}	
\end{figure}


\vero{In} \citep{bacro2016flexible} daily rainfall data in the East of Australia \vero{are studied. Different models (MS, AI, and MM) are fitted to the data. It is showed} that MM models has 
the merit to overcome the limits of MS 
models in which only AD or exact AI can be modeled.
\\

\section{$F^{\lambda}$-madogram for spatial max-mixture model}\label{sec:mado}

In the present paper we \vero{shall use the} $F^{\lambda}$-madogram \vero{defined by Equation (\ref{lambda-madogram})} for max-mixture models (\ref{max-mixture}). \vero{We begin by calculating the 
	expression of the $F^\lambda$-madogram for max-mixture models. Then, we shall develop a $F^\lambda$-madogram procedure to estimate $\theta_X$, $\theta_Y$ and chose $a$. }\\

\begin{prop}\label{prop:f_lambda} Let $\{X(s), s \in \mathcal{S}\}$ be a \vero{simple} max-stable process, with extremal coefficient function $\theta_{X}(h)$, and $Y(s)$ be an inverted max-stable 
	process with coefficient of tail dependence function $\eta(h)= 1/\theta_{Y}(h)$. \vero{Let $a\in[0\/,1]$ and $Z = \max\{aX\/, (1-a)Y\}$. Then, the $F^\lambda$-madogram of the spatial max-mixture process 
		$Z(s)$} is given by
	
	\begin{eqnarray}
		\lefteqn{\nu_{F^\lambda}(h)=\frac{\lambda}{1+\lambda} - \frac{2\lambda}{a (\theta_{X}(h) -1)+1+\lambda}+}\nonumber \\
		&&\frac{\lambda}{ a \theta_{X}(h) +\lambda} -\frac{\lambda \theta_{Y}(h) 
		}{(1-a) \theta_{Y}(h) +a \theta_{X}(h)+\lambda}\beta \left(\frac{a \theta_{X}(h)+\lambda}{ 1-a}, \theta_{Y}(h) \right)
		\label{madogram formula}
	\end{eqnarray}
	where $\beta(.,.)$ is the beta function.
\end{prop}
\raggedbottom
\vero{As a consequence of Proposition \ref{prop:f_lambda}, we easily recover the expressions of $F^\lambda$-madograms for max-stable and inverted max-stable processes. The $F^\lambda$-madogram of a 
	simple max stable process $X$ with extremal dependence coefficient $\theta_X$ is:}
\begin{equation}\label{eq:f_lambda_MS}
	\nu_{F^\lambda}(h)= \frac{\lambda}{\lambda+1}\frac{\theta_{X}(h)-1}{\lambda+\theta_{X}(h)}
\end{equation}
\vero{and we have} $\nu_{F^\lambda}(h) \in [0, \frac{\lambda}{(1+\lambda)(2+\lambda)}]$. \vero{The $F^\lambda$-madogram of an inverted max-stable process $Y$ with extremal dependence coefficient 
	$\theta_Y$  is:}
\begin{equation}\label{eq:f_lambda_IMS}
	\nu_{F^\lambda}(h)= \frac{1}{1+\lambda}-\frac{\lambda \theta_{Y}(h)}{\lambda+ \theta_{Y}(h)} \beta(\lambda,\theta_{Y}(h))
\end{equation}
\begin{proof}{\em of Proposition \ref{prop:f_lambda}} \vero{We use that for any $x\/,y\in\R$, $|x-y|= 2$ max $ \{x,y\}-(x+y)$ with $x=F^{\lambda}\{Z(s)\}$ and $y=F^{\lambda}\{Z(s+h)\}$. Moreover, 
		recall that} $\mathbb{E}[F^{\alpha}\{Z(s)\}]= 1 / (1+\alpha)$. We have
	
	\begin{align*}
		\nu_{F^\lambda}(h)&= \mathbb{E} \left[\max \{F^{\lambda}{\{Z(s)\}} , F^{\lambda}{\{Z(s+h)\}}\}\right] - \frac{1}{2}\mathbb{E}\left[F^{\lambda}{\{Z(s)\}} + F^{\lambda}{\{Z(s+h)\}}\right]\\
		&= \mathbb{E} \left[\max \{F^{\lambda}{\{Z(s)\}} , F^{\lambda}{\{Z(s+h)\}}\}\right] - \frac{1}{(1+\lambda)}
	\end{align*}
	
	\vero{Consider} the random variable \vero{$W=\max \{F^{\lambda}{\{Z(s)\}} , F^{\lambda}{\{Z(s+h)\}}\}$}. Then probability distribution \vero {function} $G$ of $W$ \vero{satisfies}
	
	\begin{align*}
		G(z)&=\mathbb{P}[W \leq z]\\
		&=\mathbb{P} \left[\max \{F^{\lambda}{\{Z(s)\}} , F^{\lambda}{\{Z(s+h)\}}\} \leq z\right]\\
		&= \mathbb{P} \left[Z(s)\leq F^{-\lambda} (z) , Z(s+h)\leq F^{-\lambda}(z)\right]\\
		&=  \mathbb{P} \left[Z(s)\leq -\frac{\lambda}{\log(z)} , Z(s+h)\leq -\frac{\lambda }{\log(z)} \right]\\ &=\exp\left\{-V_{X}\left(-\frac{\lambda}{a \log(z)},-\frac{\lambda}{a \log(z)}\right)\right\} 
		\cdot \left[ z^{\frac{1-a}{\lambda}} + z^{\frac{1-a}{\lambda}} -1 \right.+ \\
		&\exp \left. \left\{-V_{Y}\left(-\frac{1}{\log\left(1-z^{\frac{1-a}{\lambda}}\right)},-\frac{1}{\log\left(1-z^\frac{1-a}{\lambda}\right)}\right)\right\}\right]
	\end{align*}
	\vero{This rewrites:}
	\begin{align*}
		G(z)&=z^{\frac{a \theta_{X}(h)}{\lambda} }\cdot  \left[ 2 z^{\frac{(1-a)}{\lambda}} -1 + \left(1-z^{\frac{(1-a)}{\lambda}}\right)^{\theta_{Y}(h)}\right]\\
		&=2 z^{\frac{a(\theta_{X}(h) -1)+1}{\lambda}  } - z^{\frac{a}{\lambda} \theta_{X}(h)} +z^{\frac{a}{\lambda} \theta_{X}(h)}   \left(1-z^{\frac{(1-a)}{\lambda}}\right)^{\theta_{Y}(h)}\/.\\
	\end{align*}
	
	\vero{Thus, we are led to}
	\begin{align*}
		\mathbb{E} [W]&= \int_{0}^{1}  z  dG\\
		&= z G \Big|_0^1 - \int_{0}^{1} G dz\\
		&=1 - \left[\frac{2\lambda}{a (\theta_{X}(h) -1)+1+\lambda}-\frac{\lambda}{ a \theta_{X}(h) +\lambda} +I \right]
	\end{align*}
	and 
	\begin{align*}
		I:= \int_{0}^{1} z^{\frac{a}{\lambda} \theta_{X}(h)}\left(1-z^{\frac{(1-a)}{\lambda}}\right)^{\theta_{Y}(h)} dz&=\frac{\lambda}{(1-a)}\beta \left(\frac{a \theta_{X}(h)+\lambda}{1-a},\theta_{Y}(h)+1 \right)\\
		&= \frac{\lambda \theta_{Y}(h) }{(1-a) \theta_{Y}(h) +a \theta_{X}(h)+\lambda}\beta \left(\frac{a \theta_{X}(h)+\lambda}{ 1-a},\theta_{Y}(h) \right)
	\end{align*}
	
	This proves Equation (\ref{madogram formula}).
\end{proof}


\vero{Figure \ref{Theortical madogram} displays} the behavior of \vero{the} $F^{\lambda} -$madogram for three different max-mixture models \vero{with respect to} the distance $h$ \vero{for} different 
$\lambda \in \{0.5, 1, 1.5, 3\}$. $\lambda =1.5$ corresponds to the largest values \vero{of the $F^\lambda$ madogram}. Figure \ref{TheorticalApp} (\nameref{sec:Appendix.A}) shows the curves of theortical 
$F^{\lambda} -$madogram for \vero{several other} max-mixture processes.

\begin{figure} [H]
	
	\includegraphics[width=0.99\linewidth, height=6cm]{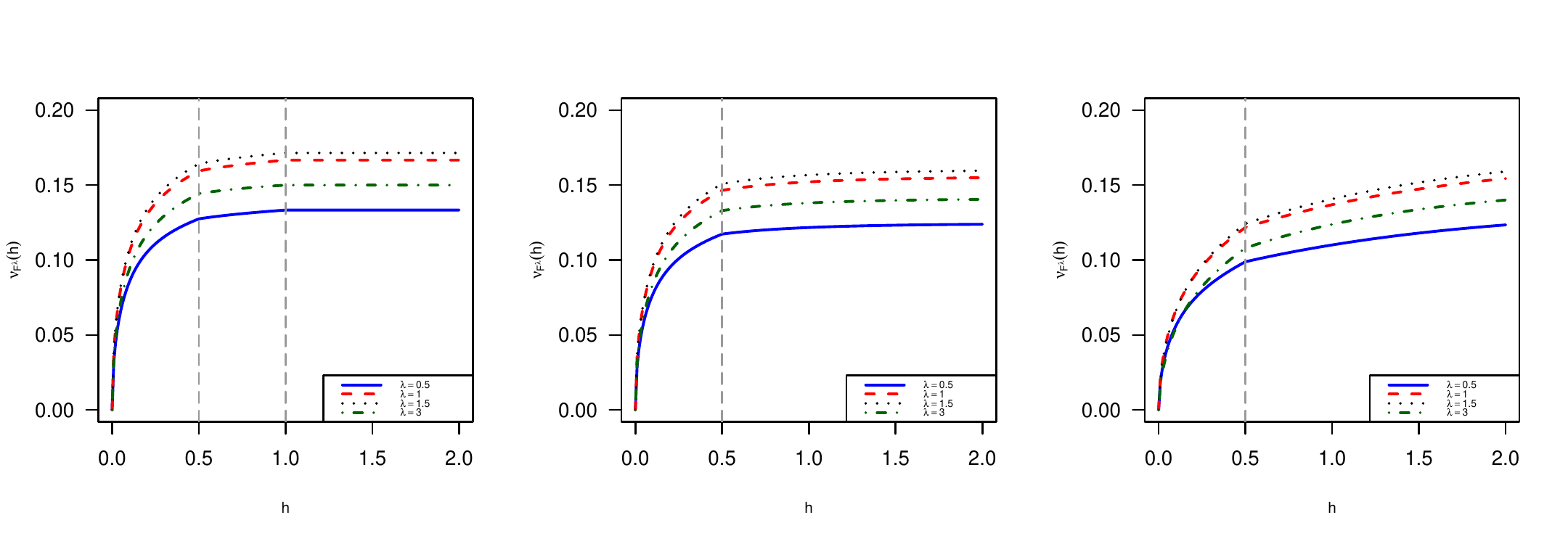}
	
	\caption{Theoretical $F^{\lambda} -$madogram functions \ref{madogram formula} \vero{for} $\lambda=0.5, 1, 1.5, 3$. \textbf{Left panel}: Max-mixture model in which $X$ is a TEG process with $\mathcal{A}_{X}$  a disk 
		of fixed radius $r_{X}=0.25$ and $\rho_{X}(h)=\exp(- h/0.2)$. The asymptotic independent process $Y$ is an inverse TEG process with $\mathcal{A}_{Y}$ a disk of fixed radius $r_{Y}=0.5$ and 
		$\rho_{Y}(h)=\exp(- h/0.4)$. \textbf{Middle panel}: Max-mixture model in which $X$ is a TEG process as \vero{before}. The asymptotic independent process $Y$ is an inverse extremal$-t$ process with  
		\vero{$v=1$ } degrees of freedom and $\rho_{Y}(h)=\exp(- h/0.6)$.  \textbf{Right panel}: Max-mixture model in which $X$ is a TEG process as \vero{before}. \vero{$Y$ is an inverted} Brown-Resnick 
		process with semivariogram $\gamma(h)= h^{2}$. The mixing 
		coefficient is setted to $a=0.5$ in the three models. The grey vertical lines represent the diameters of the disks for TEG processes.}
	
	\label{Theortical madogram}
	
\end{figure}

\vero{We are now in position to describe a choice scheme for the mixing parameter $a$ of a max-mixture process $Z$ in (\ref{max-mixture}).  From Equation 
	(\ref{madogram formula}), we may write the $F^\lambda$-madogram as a function of $a$, $\lambda$, $\theta_X$ and $\theta_Y$, that is $ \nu_{F^\lambda}(h) =  
	\Phi(a,\lambda,\theta_{X}(h),\theta_{Y}(h))$. The idea of our choice procedure is that  $\theta_X$ and 
	$\theta_Y$ may be estimated by $\widetilde{\theta}_X$ and $\widetilde{\theta}_Y$, minimizing the square difference between $\Phi(a,\lambda,\theta_{X}(h),\theta_{Y}(h))$ and its empirical counterpart, 
	then we can choose $a$ such that the empirical version of the $F^{\lambda'}$-madogram is the closest to 
	$\Phi(a,\lambda',\widetilde{\theta}_X(h),\widetilde{\theta}_Y(h))$. This idea is close to the non parametric estimation of the parameters of MM processes that has been proposed in \citep{manaf} as an 
	alternative to maximum composite likelihood estimation.\\
	Formally, we consider $Z_{i}\/, i= 1,...,N$  copies of $Z$,
	$$Q_{i} (h,\lambda)= \frac{1}{2}|F^{\lambda}(Z_{i}(s)) -F^{\lambda}(Z_{i}(s+h))|\/,$$
	where $F$ denotes the unit Fréchet distribution function. From the definition of the $F^\lambda$-madogram, we have $\mathbb{E} [Q_{i} (h,\lambda)]=\nu_{F^\lambda}(h)$}. Denote by $\Lambda \subset [0,\infty)$ a finite set of some possible $\lambda$ choices, then for a given value of $a$, a semi-parametric nonlinear least squares minimization procedure for estimating the extremal coefficient $\boldsymbol{\theta}(h) =(\theta_{X} (h),\theta_{Y} (h))^t$ is 
\begin{equation} \label{eq:theta}
	\boldsymbol{{\widetilde{\theta}}^a_{NLS}}(h)=\underset{\theta \in [1\/,2]^2} {\text{arg min}} \quad N^{-1} \sum_{\lambda \in \Lambda}  \sum_{i=1,...,N}\left[{Q_{i}(h\/,\lambda)- 
		\Phi(a\/,\lambda\/,\theta_{X}(h),\theta_{Y}(h))}\right]^{2}	
\end{equation}

\vero{Assume that the $Z_i$'s are observed at locations $s_1\/,\ldots\/, s_K$ and let $h$ be the pairwise distances between the $s_j$'s.} \veron{We shall denote by $\widehat{\nu}_{F^{\lambda}} (h)$ 
	the empirical 
	version of $\nu_{F^\lambda}(h)$, that is for $\Vert s_\ell-s_p\Vert = h$,
	$$\widehat{\nu}_{F^{\lambda}} (h)= \frac{1}{2 N} \sum_{i=1}^{N} |F^{\lambda}\{Z_{i} (s_\ell)\} - F^{\lambda}\{Z_{i} (s_p)\}|\/.$$}
\vero{For $a$ fixed, let $\boldsymbol{{\widetilde{\theta}}^a_{NLS}}(h)=(\widetilde{\theta}_X^a(h)\/,\widetilde{\theta}_Y^a(h))^t$ be estimated as above with some chosen distinct values 
	$\lambda \in \Lambda$. Let $\lambda' \notin \Lambda$ and denote by $\mathcal{H}\subset [0,\infty)$ a finite set of spatial lags $h$. Moreover, let $\widetilde{\nu}_{F^\lambda}$ be the estimation of $\nu_{F^\lambda}$ where $\theta_X$ and $\theta_Y$ are replaced in 
	(\ref{madogram formula}) by $\widetilde{\theta}_X^a$ and $\widetilde{\theta}_Y^a$.  We define:}
\veron{
	\begin{equation}\label{eq:DC}
		\text{DC}(a) = \sum_{h \in \mathcal{H}}\omega(h)\left[\frac{\displaystyle \widehat{\nu}_{F^{\lambda'}}(h)}{\displaystyle \widetilde{\nu}_{F^{\lambda'}}(h)}-1 \right]^2
	\end{equation}
}
\vero{
	where the $\omega{(h)}$'s are nonnegative weights which can be used for example to reduce the number of pairs included in the estimation. A simple choice for these weights is  $\omega{(h)}= 
	\mathbbm{1}_{\{ h \leq r\}}$ where the $r$ value can be chosen as the $q-$quantile of the distributions of 
	the distances $h$ between pairs of sites, $q \in (0,1) $. Finally, we \veron{select} $a$ that gives the lower value of (DC).\\
	\\
	Of course, when dealing with real data, the marginal laws are usually not unit Fréchet and thus have to be changed to unit Fréchet. In that case, the empirical distribution function $\widehat{F}$ is 
	used instead of $F$ in the definitions of $Q_i(h\/,\lambda)$ and $\widehat{\nu}_F^{\lambda}$:\veron{
		\begin{eqnarray*}
			\widehat{Q_i}(h\/,\lambda)& = &\frac12 |\widehat{F}^\lambda(Z_i(s))- \widehat{F}^\lambda(Z_i(s+h))| \/, \ \mbox{and} \\ 
			\widehat{\widehat{\nu}}_{{F}^{\lambda}} (h)&=& \frac{1}{2 N} \sum_{i=1}^{N} |\widehat{F}^{\lambda}\{Z_{i} (s_\ell)\} - \widehat{F}^{\lambda}\{Z_{i} (s_p)\}|\/.
		\end{eqnarray*}
		We shall denote $\widehat{\theta}_X^a$, $\widehat{\theta}_Y^a$, $\widehat{a}$ 
		the estimations of $\theta_X$, $\theta_Y$ and $a$ obtained when $\widehat{Q_i}(h\/,\lambda)$ and $\widehat{\widehat{\nu}}_{F^{\lambda}}$ are used.}} 
\\
In order to get the consistency of our estimations, we need the two following assumptions:
\begin{itemize}
	\item $I_1$: for any $a\in[0\/,1]$,$\{\lambda_{1}\/,\lambda_{2}\} \in \Lambda$ with $\lambda_1\neq \lambda_2$, the mapping
	\begin{eqnarray*}
		[1\/,2]^2 & \longrightarrow & \R^2\\
		(x\/,y) & \mapsto & (\Phi(a\/,\lambda_1\/,x\/,y)\/, \Phi(a\/,\lambda_2\/,x\/,y))
	\end{eqnarray*}
	is injective. 
	\item $I_2$: let $\theta_X$, (resp. $\theta_Y$), $\theta_X'$, (resp. $\theta_Y'$) be extremal coefficients of max-stable (resp. inverse max-stable) processes. Let $\lambda$ be fixed, if for all $h$,
	$\Phi(a\/,\lambda\/,\theta_X(h)\/,\theta_Y(h)) = \Phi(a'\/,\lambda \/,\theta_X'(h)\/,\theta_Y'(h))$ then $a=a'$. 
\end{itemize}
\begin{remark}
	The hypothesis $I_1$ and $I_2$ are identifiability hypothesis. Numerical tests on several models seem to indicate that they are satisfied for various max-mixtures models but we did not succeed to 
	prove it. 
\end{remark}
\begin{theorem}\label{a_consist}
	Assume that  $(Z_{i}(s_j))_{i= 1,...,N}$ are i.i.d copies of $Z(s_j)$, $j=1\/, \ldots\/, K$ where $Z$ is a max-mixture spatial process with mixing coefficient $a_0\in[0\/,1]$. Assume that the 
	injectivity conditions $I_1$ and $I_2$ are verified.  Then the 
	estimations of $a$ by $\widetilde{a}$ and $\widehat{a}$ are consistent in the sense that
	$$\widetilde{a} \longrightarrow a_0 \ \mbox{in probability as} \ N\rightarrow\infty\ \mbox{and} \ \widehat{a} \longrightarrow a_0 \ \mbox{in probability as} \ N\rightarrow\infty\/.$$
\end{theorem}
\begin{proof}
	We shall give the proof for  $\widehat{a}$, the proof for $\widetilde{a}$  is simpler and can be done along the same lines. We first begin by proving the consistency of $\widehat{\theta}_X^{a_0}$ and 
	$\widehat{\theta}_Y^{a_0}$. Consider $\{\lambda_1, \lambda'_1\} \in \Lambda$, with $\lambda_1 \neq \lambda'_1$. Write 
	$$\varepsilon_{h\/,i}^1 = \widehat{Q}_i(h\/,\lambda_1) - \Phi(a_0\/,\lambda_1\/,\theta_X(h)\/,\theta_Y(h)) \  \mbox{and}$$
	$$\varepsilon_{h\/,i}^2 = \widehat{Q}_i(h\/,\lambda'_1) - \Phi(a_0\/,\lambda'_1\/,\theta_X(h)\/,\theta_Y(h))\/.$$
	Using the convergence results from \citep{doi:10.1093/biomet/asp001} (Proposition 3), we have that 
	$$\frac1N \sum_{i=1}^N \varepsilon_{h\/,i}^k \rightarrow 0 \ \mbox{in probability} \/, \ k=1\/, 2 \ \mbox{and}$$
	$$\exists \ \sigma_k \geq 0 \ \mbox{such that} \ \frac1N \sum_{i=1}^N (\varepsilon_{h\/,i}^k )^2 \rightarrow \sigma_k^2 \/, \ \mbox{in probability}  \/, \  k=1\/, 2 \/.$$
	From this remark, following the lines of proof of Theorem II.5.1 in \citep{antoniadis1992regression}, see also the proof of Theorem 4.1 in \citep{manaf}, we conclude, using the injectivity hypothesis 
	$I_1$ that for any $h$, $\widehat{\theta}_X^{a_0} (h) \longrightarrow \theta_X(h)$ in probability and $\widehat{\theta}_Y^{a_0} (h) \longrightarrow \theta_Y(h)$ in probability.\\
	Now, consider $(a^*\/,\theta_X^*(h)\/,\theta_Y^*(h))$, a limit point of $(\widehat{a}\/, \widehat{\theta}_X^{\widehat{a}}(h)\/,\widehat{\theta}_Y^{\widehat{a}}(h))$. Since $\widehat{a}$ reaches the minimum of $DC$, we have $DC(\widehat{a}) \leq DC(a_0)$. The convergence of $\widehat{\theta}_X^{a_0}$, $\widehat{\theta}_Y^{a_0}$ and 
	$\widehat{\widehat{\nu}}_{F^{\lambda'}}$ from \citep{doi:10.1093/biomet/asp001} (Proposition 3) implies that $\displaystyle\lim_{N\rightarrow \infty} DC(a_0) =0$ so that $DC(a^*)=0$ which leads to 
	$\Phi(a^*\/,\lambda'\/, \theta^*_X(h)\/,\theta_Y^*(h)) = \Phi(a_0 \/,\lambda_2\/,\theta_X(h)\/,\theta_Y(h))$ with $w(h)\neq 0$ and thus $a^*=a_0$ by the injectivity 
	condition $I_2$. 
\end{proof}

\begin{remark}
	As seen by our selection scheme above, we have to determine the choices of $\lambda \in \Lambda$ and $\lambda' \notin \Lambda$. After a lot of trials and as a compromise between accuracy and computation time we found that good results can be acheived with $\Lambda =\{1,3\}$ and $\lambda'=1.5$ as shown by the simulation study, Section \ref{sec:simu}.
\end{remark}

\section{Simulation Study}\label{sec:simu}
\veron{This section is devoted to a simulation study, in order to benchmark our estimation procedure. We shall consider two different max-mixture models.\\
	\ \\}

$\mathbf{M_{1}}$ is a max-mixture model in which $X$ is a TEG process with $\mathcal{A}_{X}$ a disk of fixed radius $r_{X}$. The AI process $Y$ is an inverse TEG process with $\mathcal{A}_{Y}$  a 
disk of fixed radius $r_{Y}$. For simplicity, we choose  stationary \veron{and} isotropic exponential correlation functions, with range parameters $\phi_{X} ,\phi_{Y}>0$ respectively. The model 
\veron{parameters} vector is $\boldsymbol{\psi}=(r_{X},\phi_{X},a,r_{Y},\phi_{Y} )^t$.\\



$\mathbf{M_{2}}$ is a max-mixture model where $X$ \veron{is an} isotropic Brown-Resnick process with variogram $2\gamma(h)= (h / \phi_{X})^{\tau}$ and $Y$ is \veron{an} isotropic inverted 
extremal$-t$ \veron{process with $v$ degrees of freedom and exponential correlation function $\rho_{Y}(h)=\exp(- h/\phi_{Y})$, $\phi_{Y} >0$. The model parameters vector is 
	$\boldsymbol{\psi}=(\phi_{X}, \tau, a, v, \phi_{Y} )^t$.}\\




We summarize our simulation study procedure in the following steps.\\

\textbf{Step 1.}  For each experiment, we consider a moderately sized data from  the \veron{two} max-mixture models described above with a true mixing coefficient $a_{0}$, 
$K=50$ sites randomly and uniformly distributed in the square $[0,L]^{2}$, $L\in \mathbb{N}$ and $N=2000$ independent \veron{replications at each site}. Max-stable processes were simulated using 
SpatialExtremes package 
in $R$ \citep{ribatet2011spatialextremes} except \veron{for the TEG processes which have been} simulated as in \citep{davison2012geostatistics}. Each experiment was repeated $M=100$ times.\\

\textbf{Step 2.} For each data set in Step 1, we estimated the extremal dependence functions $\theta_{X}(h)$ and $\theta_{Y}(h)$ using the nonlinear least squares estimation criterion \veron{as 
	described in (\ref{eq:theta}), using $\widehat{Q}_i$.} This step was performed with a set of different mixing coefficients including the true one $a_{0}$.\\

\textbf{Step 3.} We calculate our decision criterion (DC($a$)) with the estimated $\widehat{\theta}_{X}^a(h)$ and $\widehat{\theta}_{Y}^a(h)$ functions from Step 2. It is expected 
that this criterion will lead to the minimum values when the true parameter $a$ is used to estimate $\theta_{X}(h)$ and $\theta_{Y}(h)$ in Step 2. Since choosing \veron{a} wrong mixing coefficient 
\veron{$a$ should} lead to \veron{a bad} estimation \veron{of} $\theta_{X}(h)$ and $\theta_{Y}(h)$. Equal weights $\omega(.)$ are \veron{used}. 
\\

To assess the performance of the nonlinear least squares estimator \veron{$\boldsymbol{{\widehat{\theta}}}^a_{NLS}(h)$}, we simulate data from max-mixture models as 
mentioned in Step 1. The performances of the estimators are given by the relative mean square error $\text{MSE}_\text{rel}$, \veron{for the $\frac{K(K-1)}{2}$  pairwise distances $h_j$ which are the 
	distances between sites pairs $(s_\ell\/,s_k)$: see \citep{bel2008assessing} p. 171 for a similar definition in the multivariate context }
\raggedbottom
\begin{equation}
	\text{MSE}_\text{rel}(h_{j}) = M^{-1} \sum_{i=1}^{M} \frac{(\widehat{\theta}^a(h_{j}) -\theta (h_{j}) )^{2}}{\theta (h_{j})}, \quad j=1,...,\frac{K(K-1)}{2}\/.
\end{equation}

Figures \ref{M13} - \ref{M43} display the estimation \veron{performances with respect to distance} $h$ using the mixing coefficients $a=0.75,0.5,0.25$. For other max-mixture models we have similar 
results which are presented in the Supplementary Material (see \nameref{sec:Appendix.B}). Generally, our estimation procedure appears to work well for all \veron{considered}  max-mixture models. Moreover, 
the estimation of ${\theta}_{X}(h)$ becomes more precise as the mixing coefficient value increases, and vice versa for ${\theta}_{Y}(h)$. 
\begin{figure} [H]
	
	\includegraphics[width=0.99\linewidth, height=9cm]{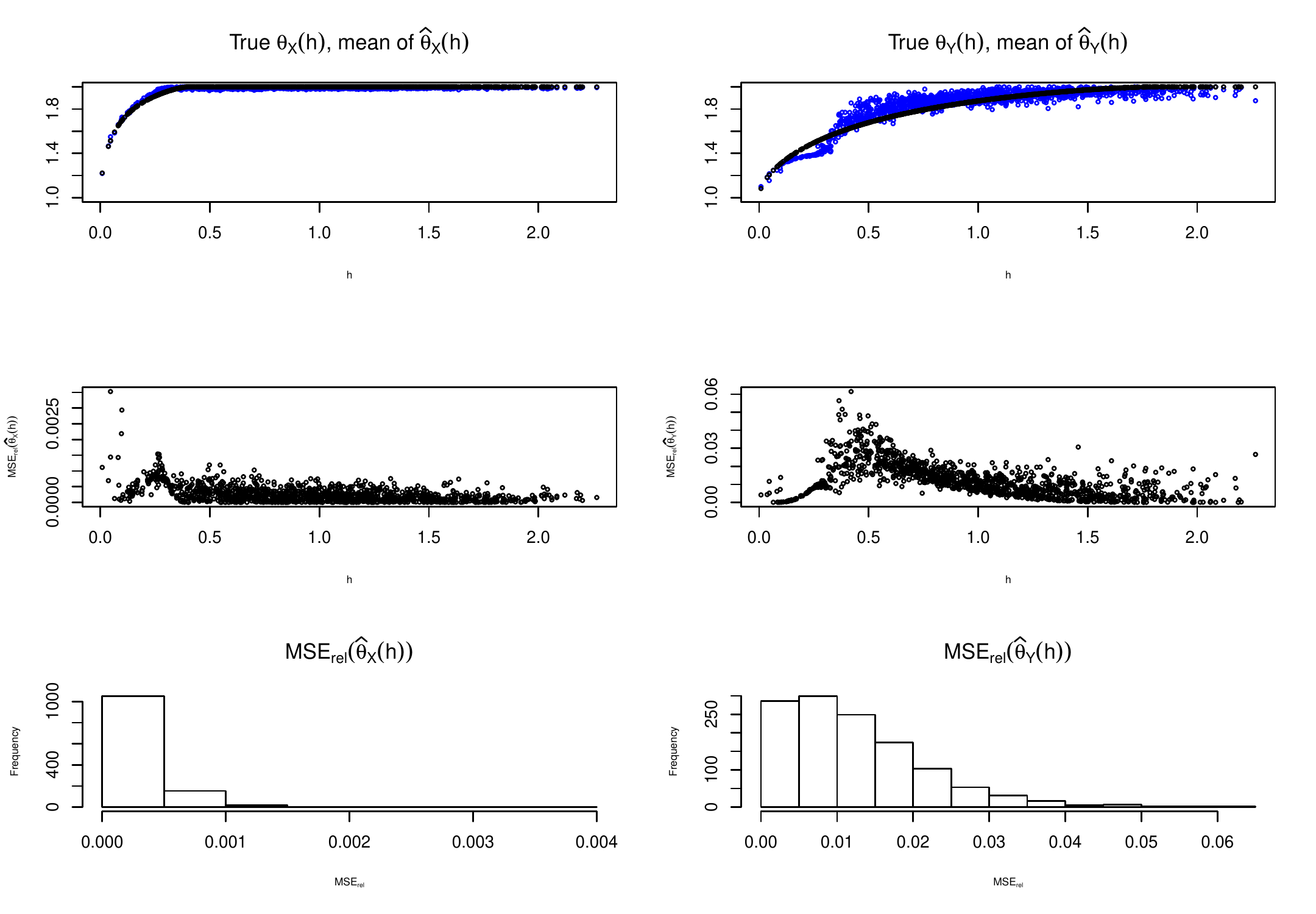}
	
	\caption{Estimation performance of $\boldsymbol{{\hat{\theta}}}_{NLS}(h)$ estimates. Data simulated from max-mixture model $\mathbf{M_{1}}$ with parameter $\boldsymbol{\psi}=(0.2,0.1,\textbf{0.75},0.9,0.7 )^t$ in the square $[0,2]^2$ as described in Step 1. The top row compares the true extermal coefficient functions (black points) and the NLS mean estimates (blue points) against the distance $h$. The middle row displays the $\text{MSE}_\text{rel}$ of the NLS estimates against the distance $h$. While the bottom row displays the histograms of $\text{MSE}_\text{rel}$ of the NLS estimates. } 
	
	\label{M11}
	
\end{figure}


\begin{figure} [H]
	
	\includegraphics[width=0.99\linewidth, height=9cm]{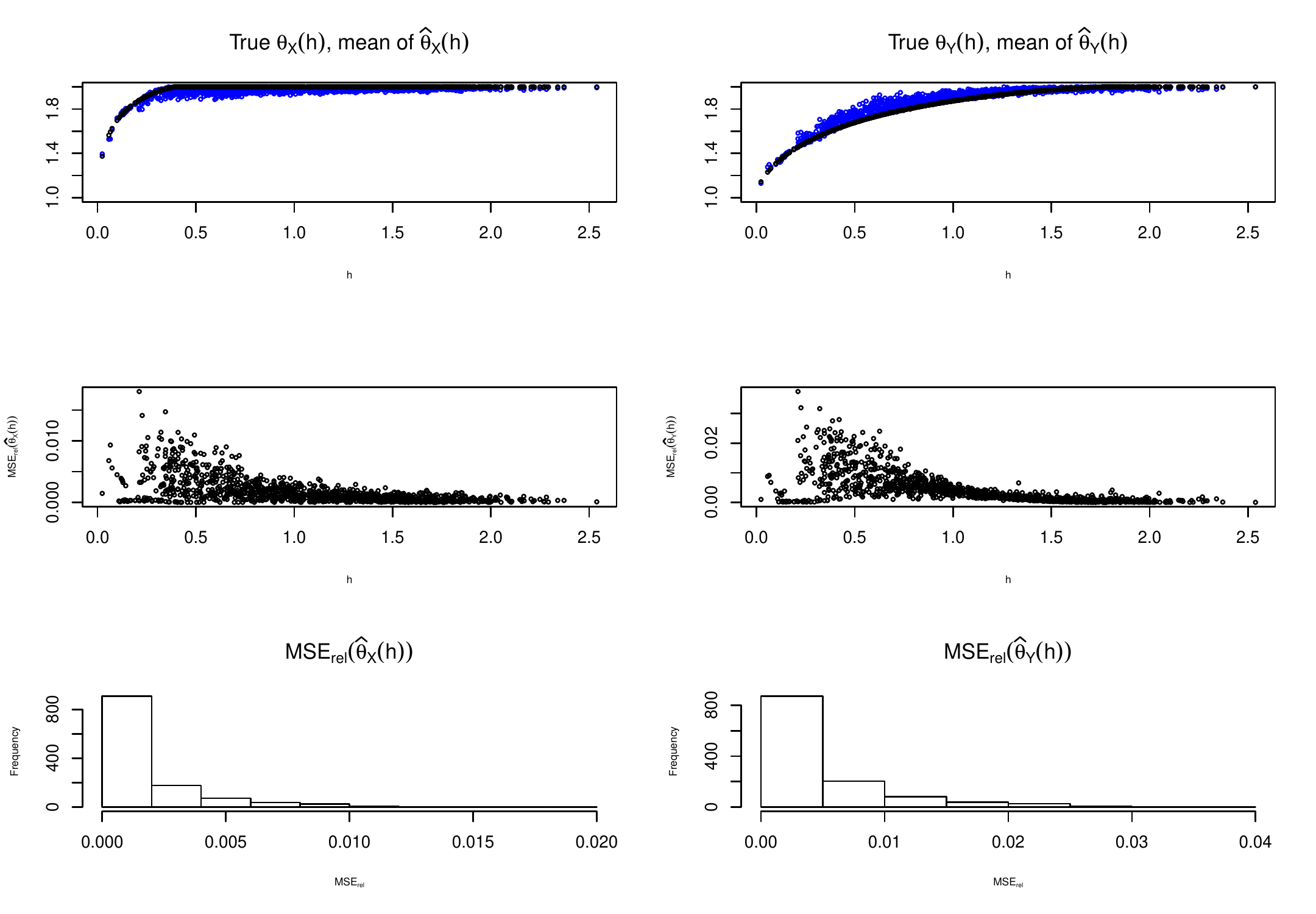}
	
	\caption{Estimation performance of $\boldsymbol{{\hat{\theta}}}_{NLS}(h)$ estimates. Data simulated from max-mixture model $\mathbf{M_{1}}$ with parameter $\boldsymbol{\psi}=(0.2,0.1,\textbf{0.5},0.9,0.7 )^t$ in the square $[0,2]^2$ as described in Step 1. The top row compares the true extermal coefficient functions (black points) and the NLS mean estimates (blue points) against the distance $h$. The middle row displays the $\text{MSE}_\text{rel}$ of the NLS estimates against the distance $h$. While the bottom row displays the histograms of $\text{MSE}_\text{rel}$ of the NLS estimates. } 
	
	\label{M12}
	
\end{figure}

\begin{figure} [H]
	
	\includegraphics[width=0.99\linewidth, height=9cm]{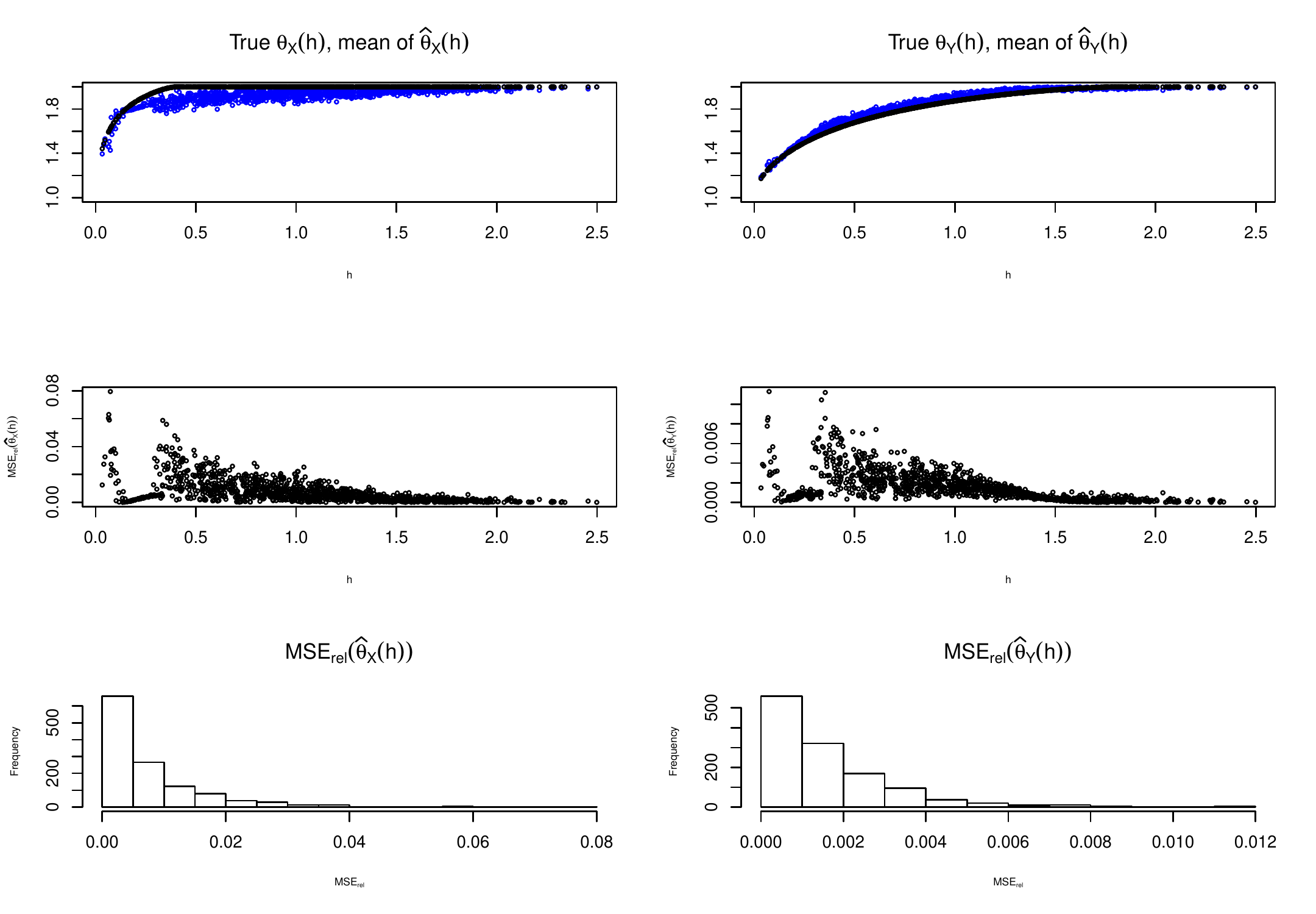}
	
	\caption{Estimation performance of $\boldsymbol{{\hat{\theta}}}_{NLS}(h)$ estimates. Data simulated from max-mixture model $\mathbf{M_{1}}$ with parameter $\boldsymbol{\psi}=(0.2,0.1,\textbf{0.25},0.9,0.7 )^t$ in the square $[0,2]^2$ as described in Step 1. The top row compares the true extermal coefficient functions (black points) and the NLS mean estimates (blue points) against the distance $h$. The middle row displays the $\text{MSE}_\text{rel}$ of the NLS estimates against the distance $h$. While the bottom row displays the histograms of $\text{MSE}_\text{rel}$ of the NLS estimates. } 
	
	\label{M13}
	
\end{figure}
%
%


\begin{figure} [H]
	
	\includegraphics[width=0.99\linewidth, height=9cm]{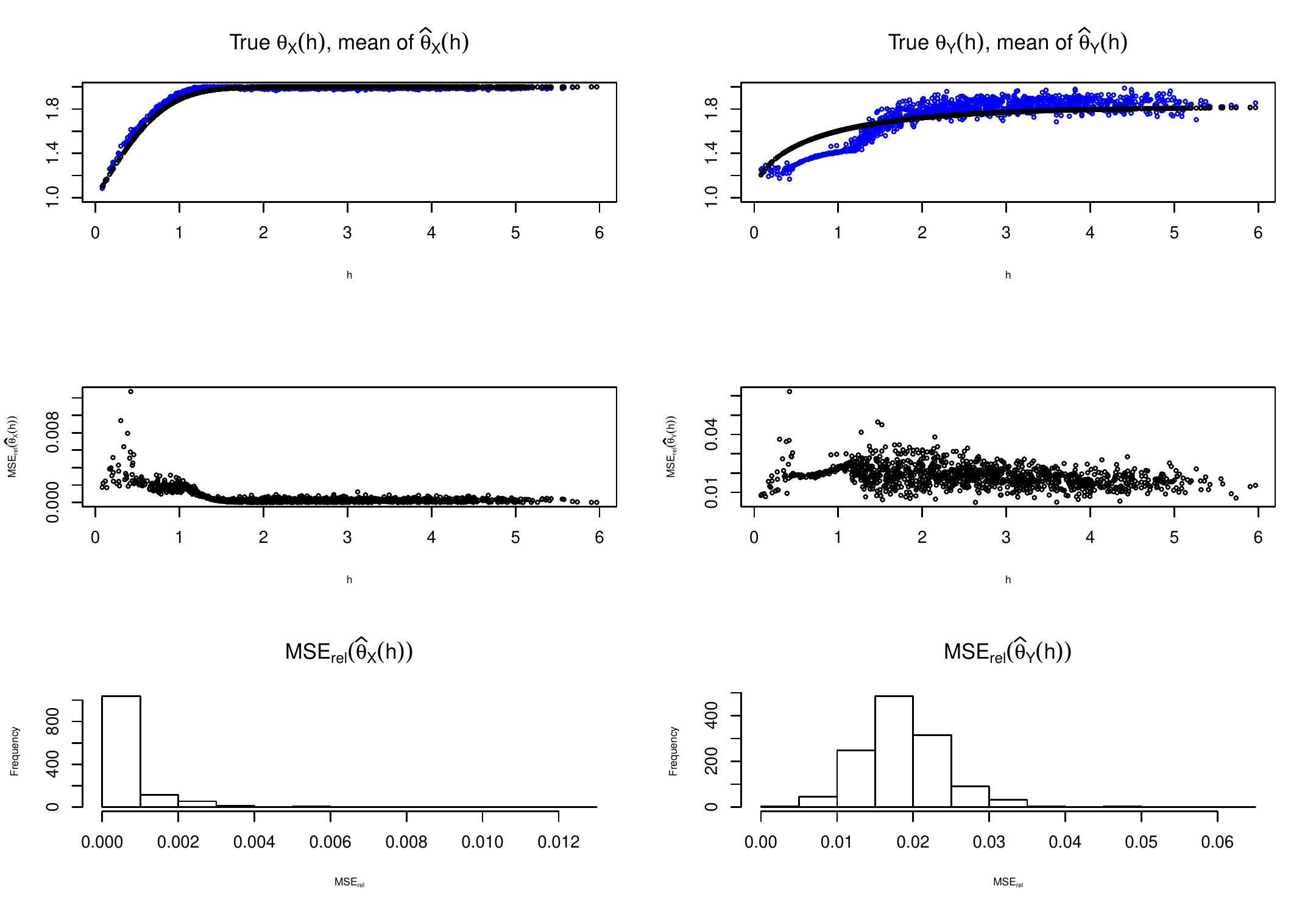}
	
	\caption{Estimation performance of $\boldsymbol{{\hat{\theta}}}_{NLS}(h)$ estimates. Data simulated from max-mixture model \veron{$\mathbf{M_{2}}$} with parameter 
		$\boldsymbol{\psi}=(0.1,2,\textbf{0.75},2,1.5 )^t$ in the square $[0,5]^2$ as described in Step 1. The top row compares the true extermal coefficient functions (black points) and the NLS mean 
		estimates (blue points) against the distance $h$. The middle row displays the $\text{MSE}_\text{rel}$ of the NLS estimates against the distance $h$. While the bottom row displays the histograms of 
		$\text{MSE}_\text{rel}$ of the NLS estimates. } 
	
	\label{M41}
	
\end{figure}

\begin{figure} [H]
	
	\includegraphics[width=0.99\linewidth, height=9cm]{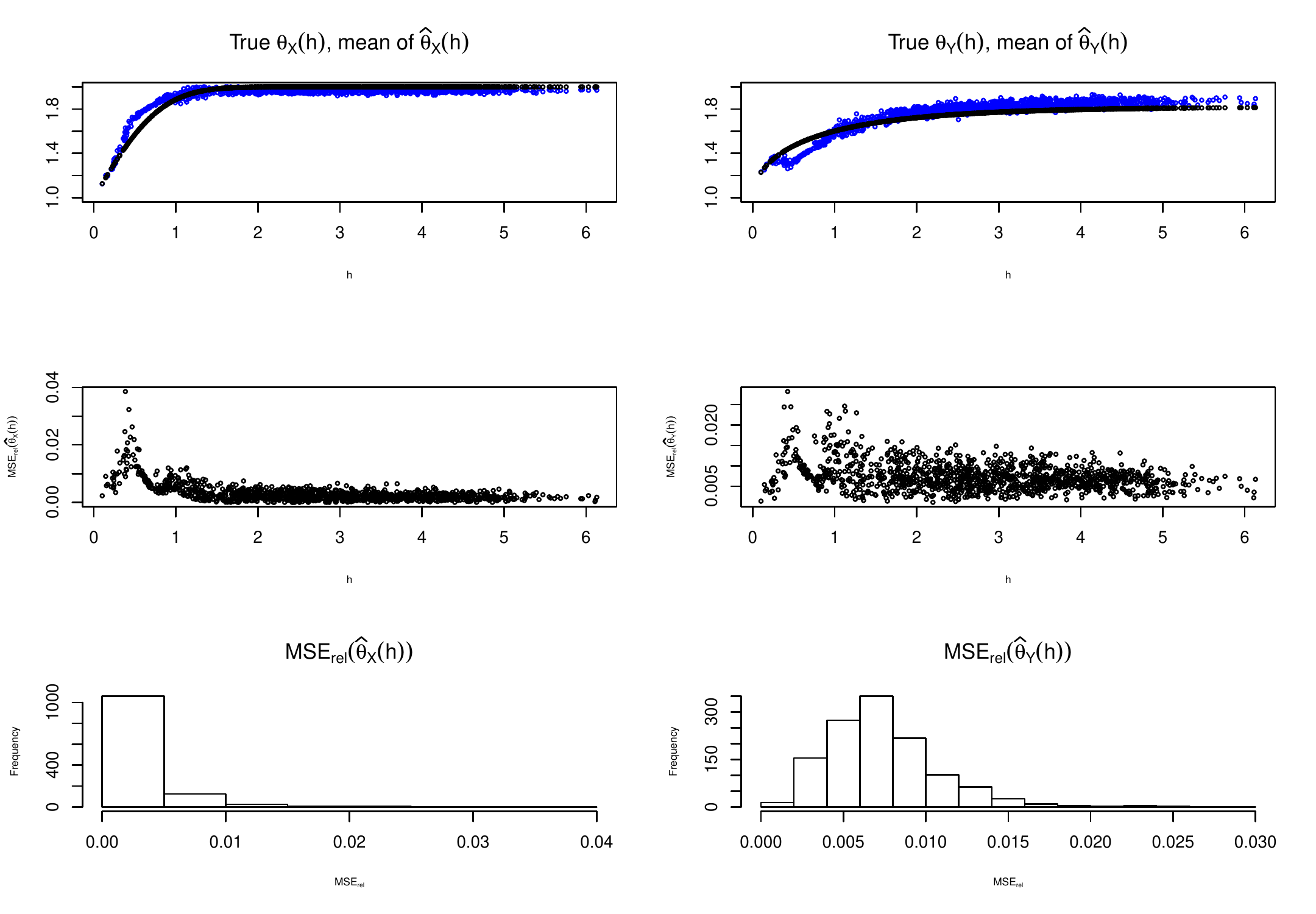}
	
	\caption{Estimation performance of $\boldsymbol{{\hat{\theta}}}_{NLS}(h)$ estimates. Data simulated from max-mixture model \veron{$\mathbf{M_{2}}$} with parameter 
		$\boldsymbol{\psi}=(0.1,2,\textbf{0.5},2,1.5)^t$ in the square $[0,5]^2$ as described in Step 1. The top row compares the true extermal coefficient functions (black points) and the NLS mean estimates 
		(blue points) against the distance $h$. The middle row displays the $\text{MSE}_\text{rel}$ of the NLS estimates against the distance $h$. While the bottom row displays the histograms of 
		$\text{MSE}_\text{rel}$ of the NLS estimates. } 
	
	\label{M42}
	
\end{figure}
\begin{figure} [H]
	
	\includegraphics[width=0.99\linewidth, height=9cm]{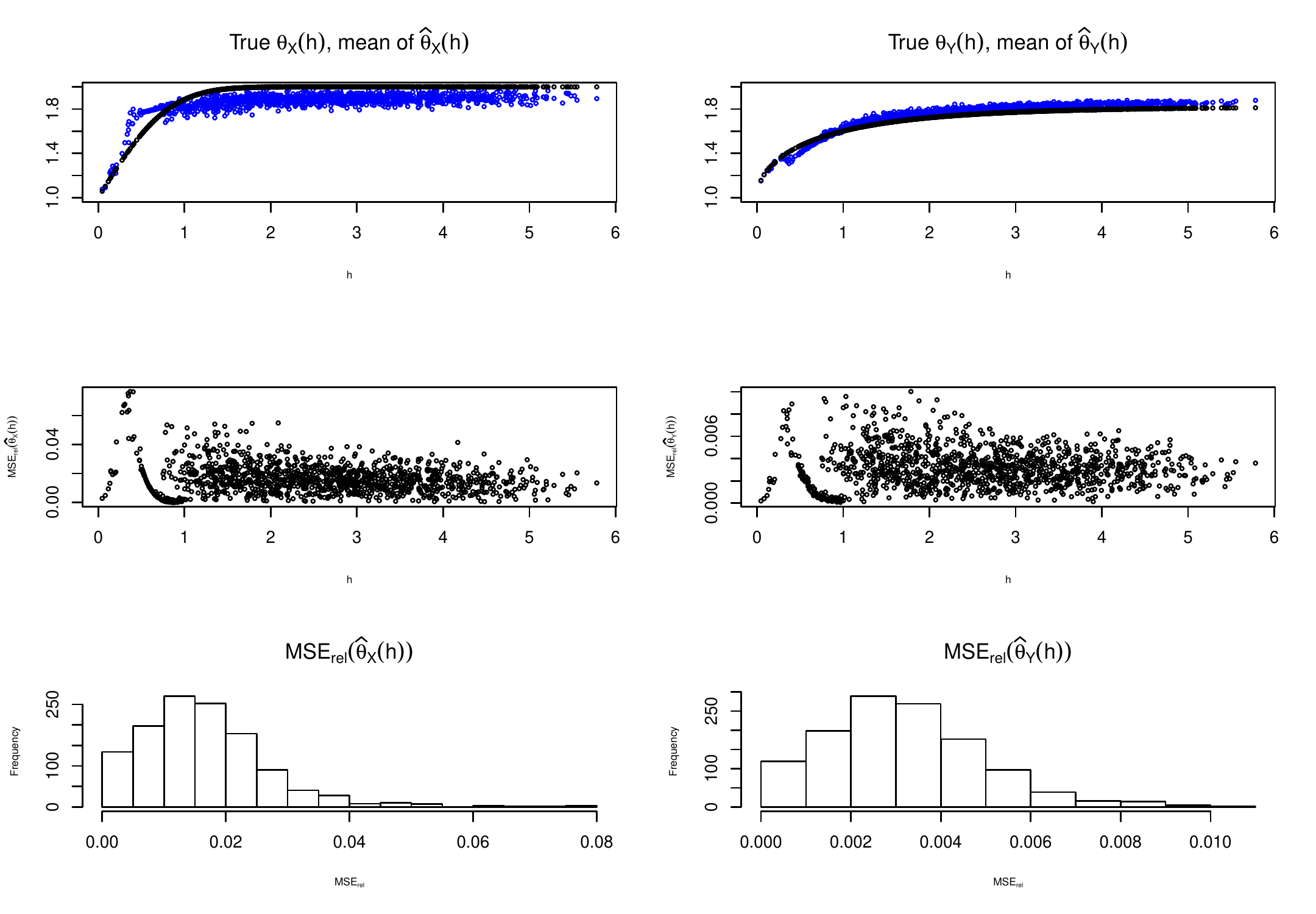}
	
	\caption{Estimation performance of $\boldsymbol{{\hat{\theta}}}_{NLS}(h)$ estimates. Data simulated from max-mixture model \veron{$\mathbf{M_{2}}$} with parameter 
		$\boldsymbol{\psi}=(0.1,2,\textbf{0.25},2,1.5 )^t$ in the square $[0,5]^2$ as described in Step 1. The top row compares the true extermal coefficient functions (black points) and the NLS mean 
		estimates (blue points) against the distance $h$. The middle row displays the $\text{MSE}_\text{rel}$ of the NLS estimates against the distance $h$. While the bottom row displays the histograms of 
		$\text{MSE}_\text{rel}$ of the NLS estimates. } 
	
	\label{M43}
	
\end{figure}
%



\veron{Now, we turn to} our proposed model selection criterion (DC) for selecting the max-mixture models with the best mixing coefficient $a$ through a number of simulation studies using the 
\veron{two} mentioned max-mixture models. The boxplots \veron{in} Figure \ref{DC1-4} display the values of the decision criterion (DC) against different mixing coefficients $a \in 
\{0,0.25,0.5,0.75,1\}$ for models $\mathbf{M_{1}}$ and \veron{$\mathbf{M_{2}}$. The lower values of DC are likely} related with the true mixing coefficient $a_{0}$. Boxplots of the decision criterion (DC) for other examples of max-mixture models in which we have similar results available in the Supplementary Material (see \nameref{sec:Appendix.B}).


\begin{figure} [H]
	
	\includegraphics[width=0.99\linewidth, height=8cm]{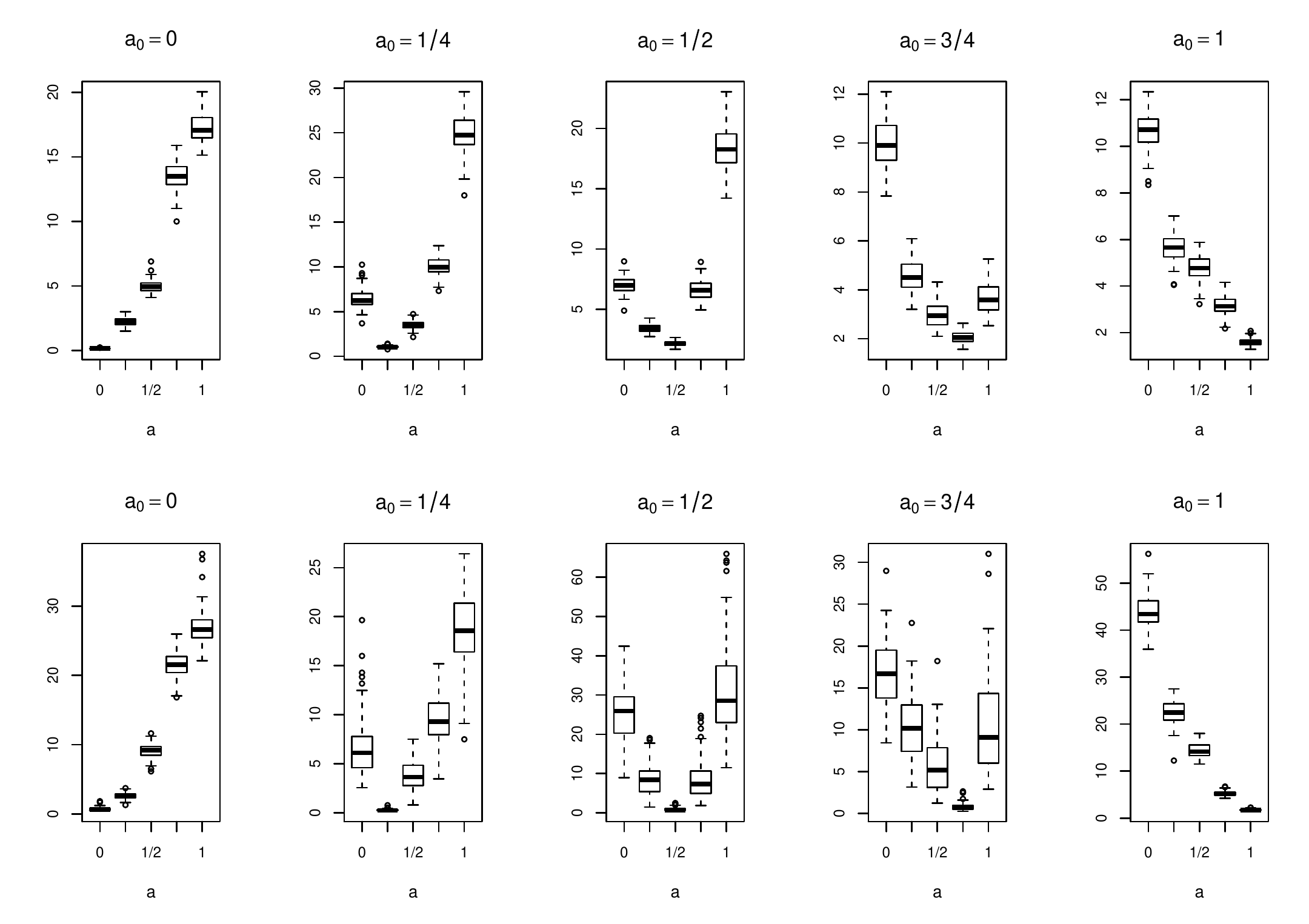}
	
	\caption{Top row: boxplots for decision criterion of 100 data replications of 2000 independent copies from $\mathbf{M_{1}}$ with parameters $\phi_{X}= 0.1$, $r_{X}= 0.2$, $\phi_{Y}= 0.7$ and 
		$r_{Y}= 0.9$ over the square $[0,2]^{2}$. Bottom row: boxplots for decision criterion of 100 data replications of 2000 independent copies from \veron{$\mathbf{M_{2}}$} with parameters  $\phi_{X}= 
		0.1$, $\tau=2$, $v=2$ and $\phi_{Y}= 1.5$ over the square $[0,5]^{2}$.}
	
	\label{DC1-4}
	
\end{figure}

\section{Real data example}\label{real_data}

The data analysed in this section are daily rainfall  amounts in (millimetres) over the years 1972-2016 occurring during April-September at 38 sites in the East of Australia whose locations are shown in Figure \ref{Map}. The altitude of the sites varying from 4 to 552 meters above mean sea level. The sites are separated by distances of approximately (34-1383)km. This data set is freely available on \href{url}{http://www.bom.gov.au/climate/data/index.shtml?bookmark=136}.
\begin{figure} [H]
	
	\includegraphics[scale=0.4]{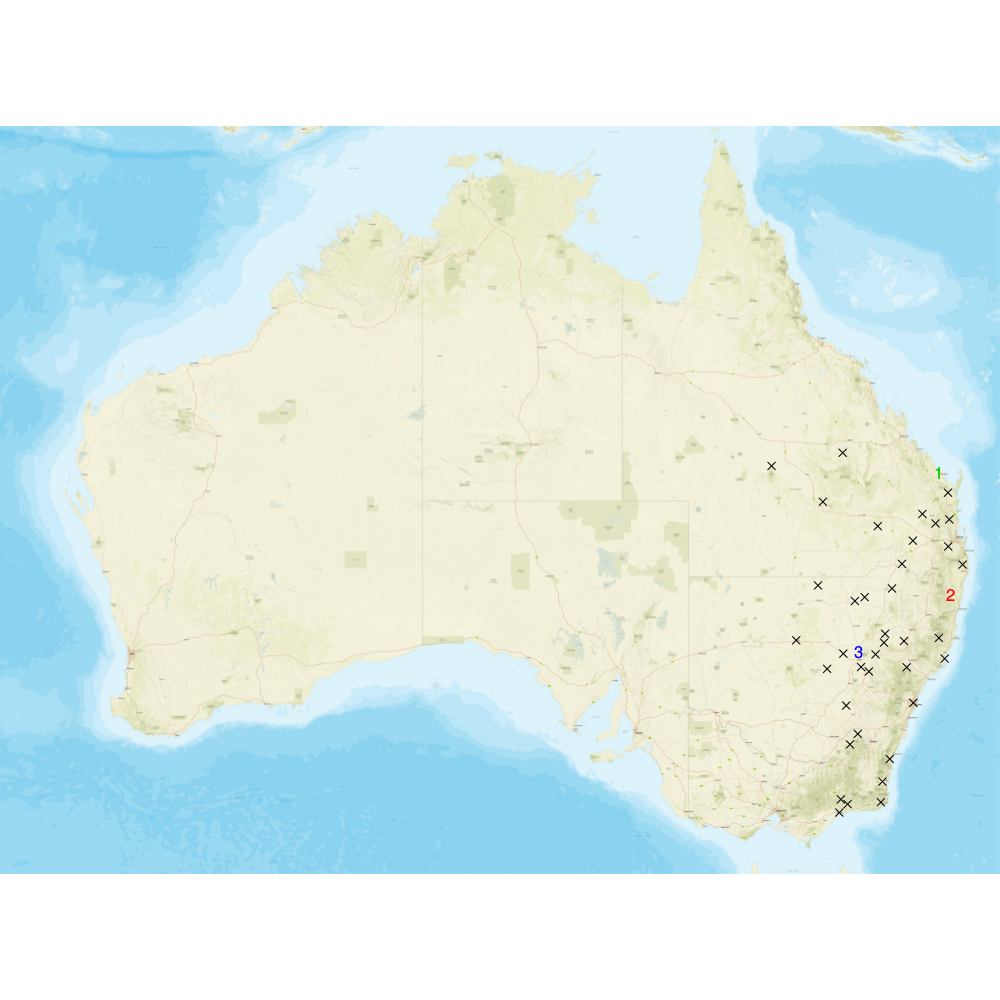}
	
	\caption{Geographical locations of  $41$ meteorological stations located in the East of Australia. Black crosses represent the 38 stations used for selecting the mixing coefficient $a$. The 
		three colored numbers \{1,2,3\} correspond to the \veron{unused} stations \veron{on which we shall predict conditional probabilities in order to validate our procedure}.}
	
	\label{Map}
	
\end{figure}

Following the approach of \citep{bacro2016flexible}  graphical assessments to explore possible anisotropy Figure \ref{Isotropy} based on the empirical estimates of the functions $\chi(h,u)$ and 
$\bar{\chi}{(h,u)}$ in different directional sectors $(- \pi/8,  \pi/8]$,  $( \pi/8,  3\pi/8]$, $( 3\pi/8,  5\pi/8]$, and $( 5\pi/8,  7\pi/8]$, where 0 represents the northing direction. Based on 
these estimates there is no clear evidence of anisotropy.

\begin{figure} [H]
	\includegraphics[width=0.9\linewidth,height=6cm]{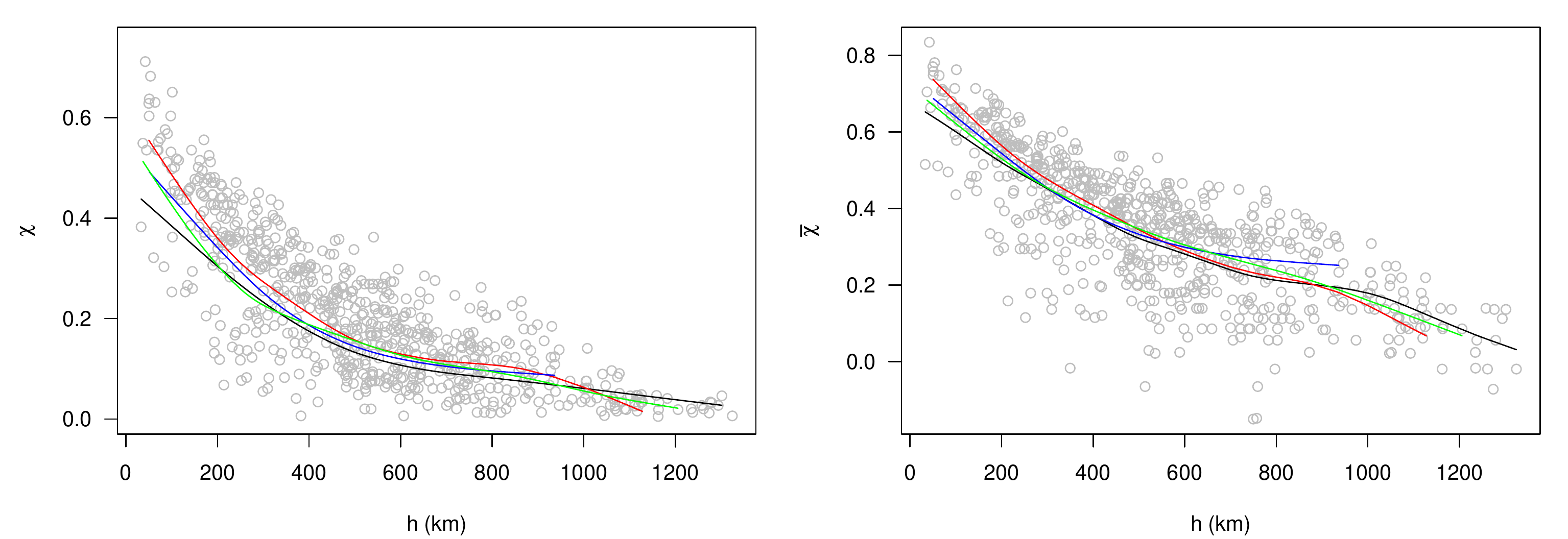}
	
	\caption{Pairwise empirical estimates of $\chi$ (left panel) and $\bar{\chi}$ (right panel) versus distance at threshold $u = 0.970$. Grey points are empirical pairwise estimates for all data pairs. Colored lines are the loess smoothed values of the empirical estimates in different directional sectors: black line $(- \pi/8,  \pi/8]$, red line $( \pi/8,  3\pi/8]$, blue line $( 3\pi/8,  5\pi/8]$, and green line $( 5\pi/8,  7\pi/8]$. }
	\label{Isotropy}	
\end{figure}
We apply our methodology for the selection of the mixing coefficient for all $a \in (0,1)$ by steps 0.01. \veron{The $a\mapsto DC(a)$ function is plotted in Figure }
\ref{result}. The best-fitting max-mixture model as judged by our DC \veron{criterium has} a mixing coefficient $a=0.34$.
\\


\begin{figure} [H]
	
	\includegraphics[width=0.5\linewidth, height=12cm]{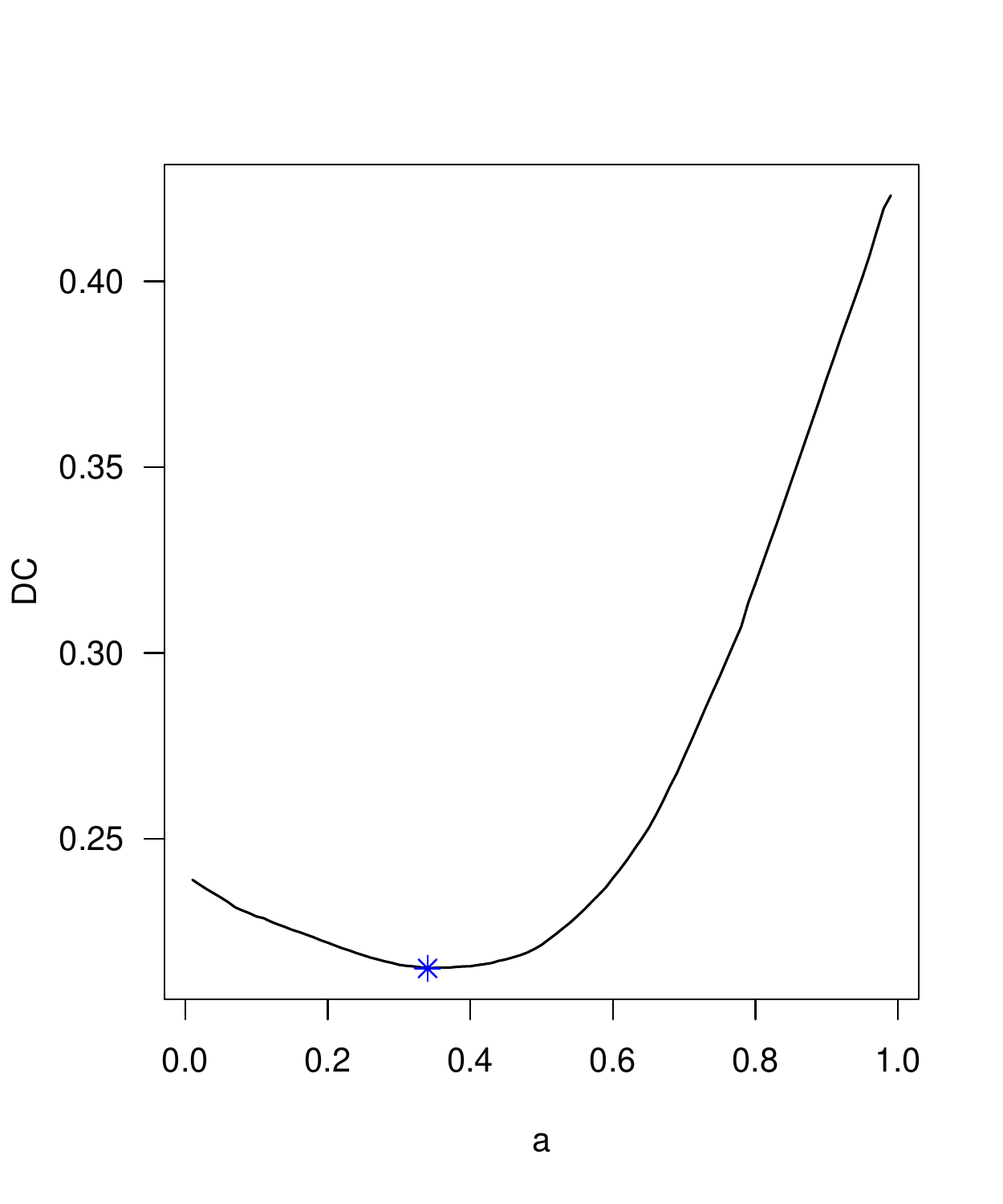}
	
	\caption{Decision criterion values for real data example \veron{on the} interval $a \in (0,1)$ by steps 0.01. The blue star corresponds to the minimum DC value which occurs at $a=0.34$.}
	
	\label{result}
	
\end{figure}

In the literature, \veron{the widely used} parametric inference procedure \veron{is} based on composite likelihood methods. In particular, pairwise likelihood estimation has been found 
usefull to estimate parameters in a max-stable process. A description of this method can be found in \citep{padoan2010likelihood, doi:10.1093/biomet/asr080, bacro2016flexible} for 
spatial context.
\\

Unfortunately, parameter estimation using composite likelihood suffers from some defects. First, it can be onerous, since the computation and subsquent optimization of the objective function is 
time-consuming. Second, the choice of good initial values for optimization of the composite likelihood is essential. Third, model dependency, a preliminary step to conduct a composite likelihood 
estimation is to specify the model that describes the dependence structure. So, this mission seems to be laborious, due to the large number of combinations that can be formed from the AD and AI 
processes stemming from max-stable processes, since the variety of dependence structures that can be assumed, i.e. changing the correlation coefficient function type in Schlather model 
\citep{schlather2002models} leads to different dependence structures. So, \veron{an} unacurate choice may lead \veron{to} severe under/over estimations of probabilities associated to simultaneous 
extreme events.\\

In the sequel, we \veron{shall estimate} the conditional probability of having daily rainfall that exceeds some  threshold $z$ at \veron{an unused} site denoted by ${s}^*_0$ 
given that this event \veron{has occurred} at the nearest observed site which \veron{is} denoted by ${s}_0$, i.e., $\mathbb{P}[Z({s}^*_0)>z|Z({s}_0)>z]$. \veron{We compare this estimation with that} obtained by the best-fitting parametric model based on composite likelihood estimation.  
\\

For this purpose, we fitted \veron{the} daily rainfall data based on censored pairwise likelihood approach used by \citep{doi:10.1093/biomet/asr080, bacro2016flexible} where the threshold is 
taken corresponding to the 0.9 empirical quantile at each site. We fitted the generalized extreme value distribution GEV$(\mu,\sigma,\xi)$ separately to each site and then data are transformed to unit Fr\'{e}chet margins through the transformation $z\rightarrow \frac{-1}{\log( \widehat{G}(z))}$, where $\widehat{G}(.)$ is the estimated GEV cumulative distribution function. The models are \\

$ \mathbf{M_{a}}$: a MM model where $X$ is a TEG process with an exponential correlation function 
$ \rho(h) = \exp (- \lVert {h} \rVert / \phi_{X})$, $\phi_{X}>0$.  $\mathcal{A}_{X}$ is a disk of fixed and unknown radius $r_{X}$, and $Y$ is an inverted TEG process with exponential 
correlation function $ \rho(h) = \exp (- \lVert {h} \rVert / \phi_{Y})$, $\phi_{Y}>0$, and $\mathcal{A}_{Y}$ is a disk with fixed and unknown radius $r_{Y}$.\\

$\mathbf{M_{b}}$: a MM model where $X$ is a TEG process as in $ \mathbf{M_{a}}$. $Y$ is an isotropic inverted Smith process where $\Sigma$ is a diagonal matrix ($\sigma_{12}=0$) with $\sigma_{11}^{2}=\sigma_{22}^{2}=\phi_{Y}^{2}$, i.e., $\gamma(h)= ( \lVert {h} \rVert / \phi_{Y})$.\\

$\mathbf{M_{c}}$: a MS TEG process described as $X$ in $ \mathbf{M_{a}}$.\\

$\mathbf{M_{d}}$: a MS isotropic Smith process where $\Sigma$ is a diagonal matrix ($\sigma_{12}=0$) with $\sigma_{11}^{2}=\sigma_{22}^{2}=\phi_{X}^{2}$, i.e., $\gamma(h)= ( \lVert {h} \rVert / \phi_{X})$.\\ 

$\mathbf{M_{e}}$: the inverted Smith process described as $Y$ in $ \mathbf{M_{b}}$.\\

The composite likelihood information criterion (CLIC) \citep{varin2005note}, defind as CLIC= $-2 [p\ell(\widehat\vartheta) - tr\{\mathcal{J}(\widehat\vartheta)\mathcal{H}^{-1}(\widehat\vartheta)\}]$ is used to 
judge the best-fitting model. Here, the maximum pairwise likelihood estimator is denoted by $\widehat\vartheta$. Lower values of CLIC indicate \veron{a} better fit. Our results are summarised in Table \ref{Composite1}. The best-fitting model for the data, as judged by CLIC, is the 
hybrid dependence model $\mathbf{M_{b}}$. 
\begin{table}[H]\centering
	
	\begin{tabular}{@{}l l l l l l l@{}}
		
		\bottomrule
		Model &${\widehat{\phi}}_{X}$  & ${\widehat{r}}_{X}$& $\widehat{{a}}$ & ${\widehat{\phi}}_{Y}$ &${\widehat{r}}_{Y}$& CLIC  \\
		
		\bottomrule
		$ \mathbf{M_{a}}$& 254.66 (179.69)& 683.32 (257.18) & 0.59 (0.21)&1609.42 (141.53) &981.73 (168.01) & 4515964 \\
		
		$ \mathbf{M_{b}}$& 93.16 (48.02)& 166.92 (80.86) & 0.27 (0.14) &971.65 (243.29)&- & ${4515911}^{*}$\\ 
		$ \mathbf{M_{c}}$ &188.49 (53.27)& 691.12 (215.53) &-&-&-&4523182 \\              
		$ \mathbf{M_{d}}$& 463.52 (216.54) & -&-&- &-&4523446\\
		$ \mathbf{M_{e}}$ &-& - &-&628.38 (86.54)  &-& 4515981\\
		\bottomrule

	\end{tabular}
	\caption{Parameter estimates of selected dependence models fitted to the daily rainfall data. The composite likelihood criterion (CLIC) and standard errors reported between parentheses. (*) indicates to the lower CLIC.}
	
	\label{Composite1}
\end{table}

\veron{Now, we shall use our least square estimations of $a$, $\theta_X$ and $\theta_Y$ in order to estimate the conditional probabilities. }
\veron{The following lemma is easily deduced from (\ref{Biv MM})}; see the proof in \nameref{sec:Appendix.C}.

\begin{lemma} 
	\label{lemma}
	\veron{Let $Z$ be a max-mixture process. Its bivariate tail distribution is given by  
		\begin{equation}
			\mathbb{P}[Z({s}^*_0)>z|Z({s}_0)>z]= \frac{1- 2 e^{-\frac{1}{z}}+ e^ {-\frac{a \theta_{X}(h_0) }{z}}\left\{-1+ 2 e ^ {-\frac{1-a}{z}}+ \left[{1-e ^ 
					{-\frac{1-a}{z}}}\right]^{\theta_{Y}(h_0)}\right\}}{1-e^{-\frac{1}{z}}}
			\label{conditional prob}
		\end{equation} 
	}
\end{lemma}

where $h_{0}$ is the separation distance between ${s}^*_0$ and ${s}_0$. \veron{Equation (\ref{conditional prob}) may be used} to estimate $\mathbb{P}[Z({s}^*_0)>z|Z({s}_0)>z]$ using both parametric and nonparametric approaches. \veron{We consider ${s}^*_0$ as the three unused} 
sites that have been marked by colored numbers \{1,2,3\} on the map Figure \ref{Map}. 
\\

The threshold $z$ is taken corresponding to the $q-$ emperical quantile at the site $s_0$, $q \in (0,1)$. For estimating $\mathbb{P}[Z({s}^*_0)>z|Z({s}_0)>z]$ nonparametrically, we fitted the data 
again by NLS procedure with the best mixing coefficient $a = 0.34$ and we obtained the estimators  $\widehat{\theta}_{X}(h_0)$ and $\widehat{ \theta}_{Y}(h_0)$  by averaging the values of 
$\widehat{\theta}_{X}(h)$ and $\widehat{ \theta}_{Y}(h)$ , where  $ h \in [h_{0}-10, h_{0}+10]$km \veroo{by taking the advantage of stationarity and isotropy of our data} and $h_{0} \in \{54.226,92.534,133.673\}$km.  While, for estimating $\mathbb{P}[Z({s}^*_0)>z|Z({s}_0)>z]$ 
parametrically, $\widehat{\theta}_{X}(h)$ and $\widehat{ \theta}_{Y}(h)$ are obtained by substituting the separating distance $h_0$ and the estimated parameters from hybrid model $ \mathbf{M_{b}}$ Table \ref{Composite1}. 
\\

In order to compare the results obtained by the \veron{two} approaches, we used the data at the \veron{unused}  three sites to compute the empirical version of conditional probabilities $\mathbb{P}[Z({s}^*_0)>z|Z({s}_0)>z]$. \veron{Below are the resulting P-P plots.} 
\\


\begin{figure} [H]
	
	\includegraphics[width=0.8\linewidth, height=8cm]{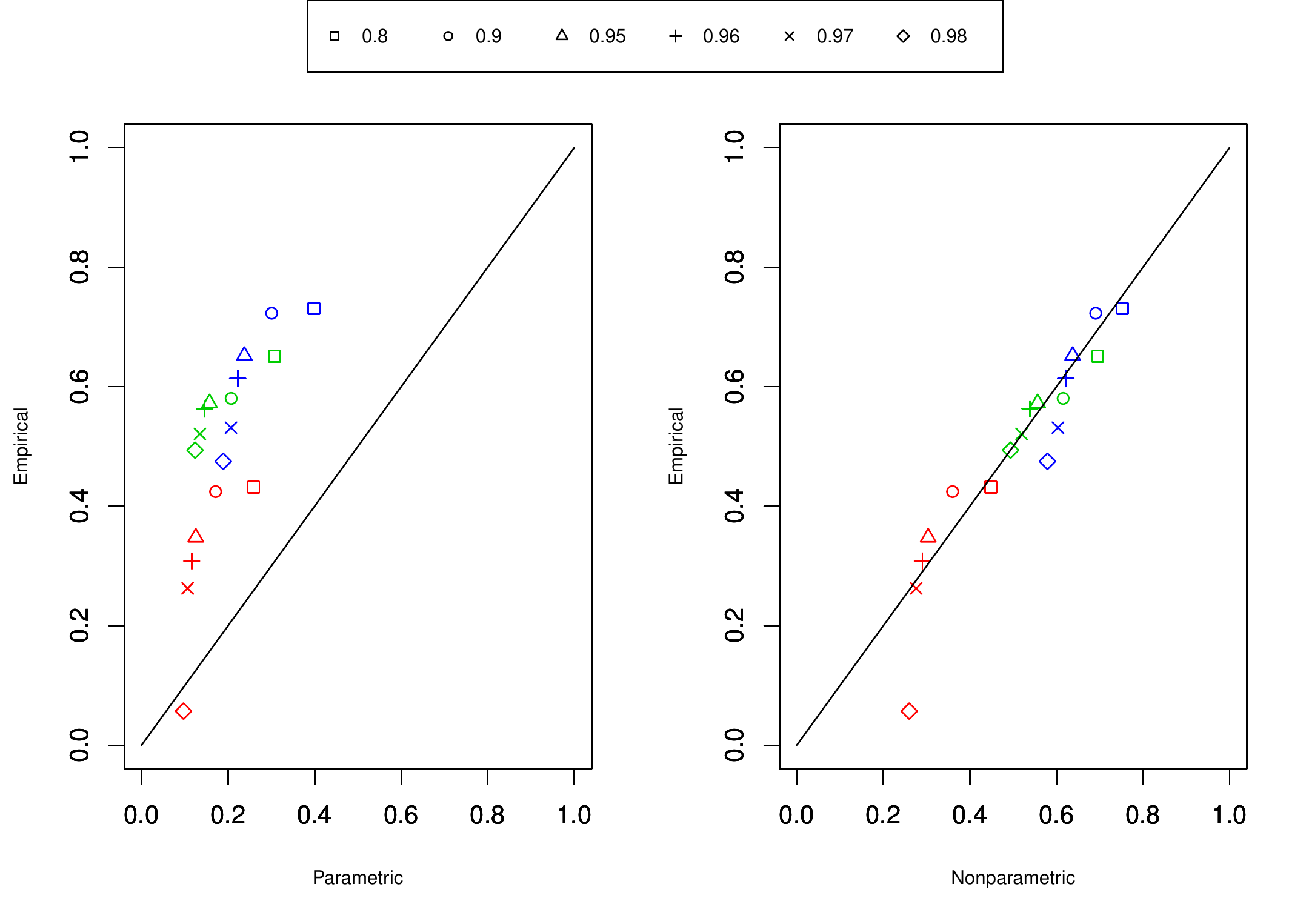}
	
	\caption{Diagnostic P-P plots for threshold excess conditional probabilities for the three \veron{unused} sites \{1,2,3\} on the map Figure \ref{Map}  obtained by both approaches. The best 
		parametric model $ \mathbf{M_{b}}$ as judged by the CLIC and our nonparametric approach. Green: site 1; red: site 2; blue: site 3.}
	
	\label{results}
	
\end{figure}

Generally, Figure \ref{results} shows that our nonparametric approach outperforms the parametric one for predicting $\mathbb{P}[Z({s}^*_0)>z|Z({s}_0)>z]$.   One of the justifications for this 
situation is that with the parametric model we have to specify a model that describes the dependence structure, and listing all choices seems to be a tedious task due to the large number of 
possibilities that can emerge from the AD and AI processes stemming from max-stable processes, i.e., different choices of correlation coefficient functions for the same model leads to different 
models. So, in this case the inaccurate choice/guess of models to be fitted may lead \veron{to} severe under/over estimation of probabilities associated to simultaneous extreme events. \\

\section{Conclusion}

In this paper, we have proposed a statistically efficient nonparametric model-free selection criterion. We can exploit our decision about the mixing coefficient to predict conditional probabilities of 
daily rainfall at unobserved sites depending on the dependence structure in the analyzed data. 
\\

The \veron{main} advantage of our approach is \veron{that it is {\em model free}}, unlike the parametric approach which assumes a specified model, so the risk of 
unaccurate choice of stochastic processes for describing the joint tail distribution may lead to severe under/over estimation of probabilities associated to 
simultaneous extreme events. 
\\

We have shown in our real data example  that the max-mixture approach appears of interest for modeling environmental data. In particular it has the eligibility to overcome the limitations of the max-stable models in which only asymptotic dependence or exact independence can be modeled.

\section*{Appendix A. }\label{sec:Appendix.A}

\begin{figure} [H]
	\includegraphics[width=0.99\linewidth, height=10cm]{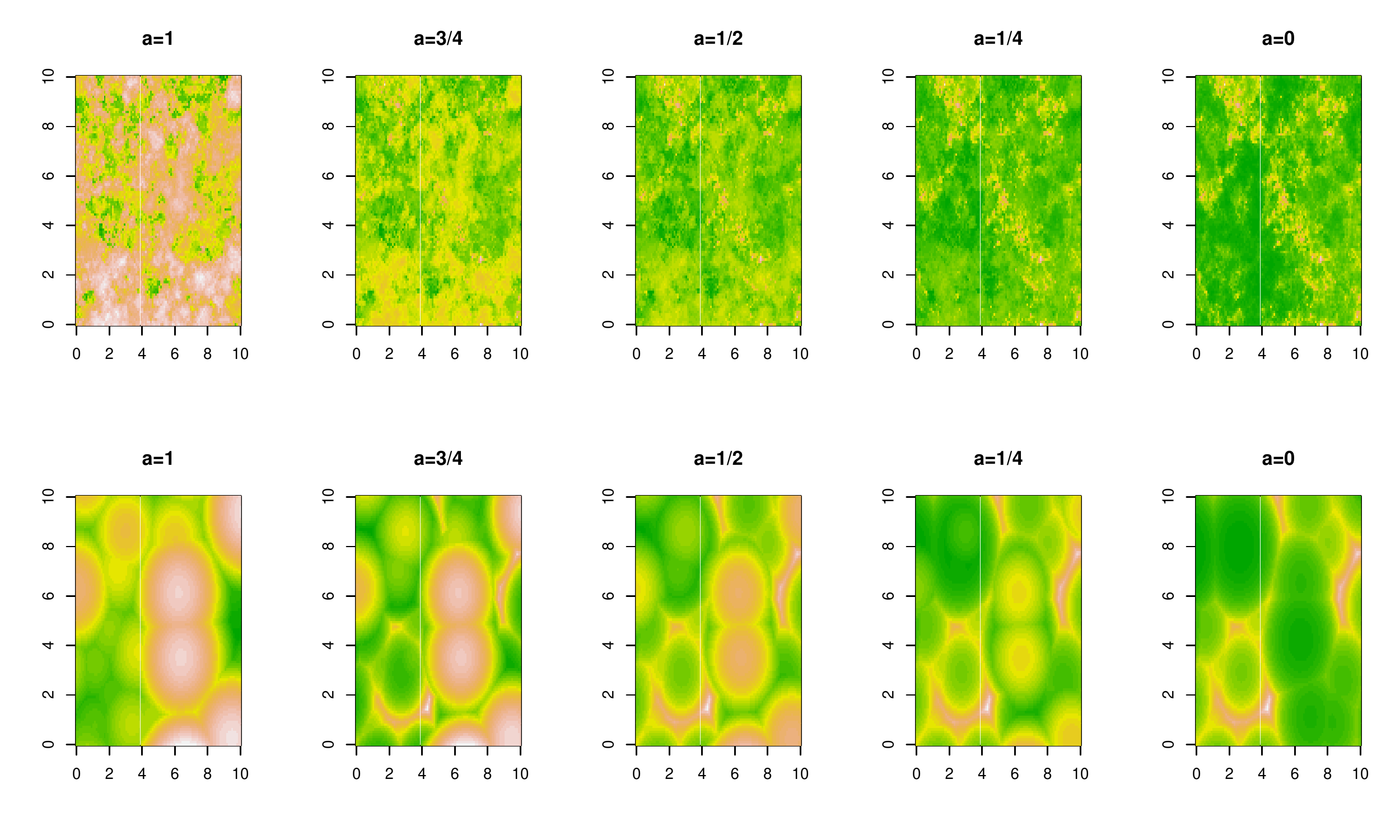}
	
	\caption{Simulations of the max-mixture model (\ref{max-mixture}) on the logarithm scale according different values of mixing coefficient $a \in \{1,0.75,0.5,0.25,0\}$. Top row: $X$ is isotropic extremal$-t$ with $v_1=1$ degrees of freedom and $\rho(h)=\exp(- h)$, $Y$ is isotropic inverted extremal$-t$ process with $v_2=2$ degrees of freedom and $\rho(h)=\exp(- h/1.5)$. bottom row: $X$ is isotropic Brown-Resnick with variogram $2\gamma(h)= h^{2}$, $X$ is isotropic inverted Brown-Resnick with variogram $2\gamma(h)= (h/1.5)^{2}$.  }
	\label{SimulationApp}	
\end{figure}

\begin{figure} [H]
	
	\includegraphics[width=0.99\linewidth, height=5cm]{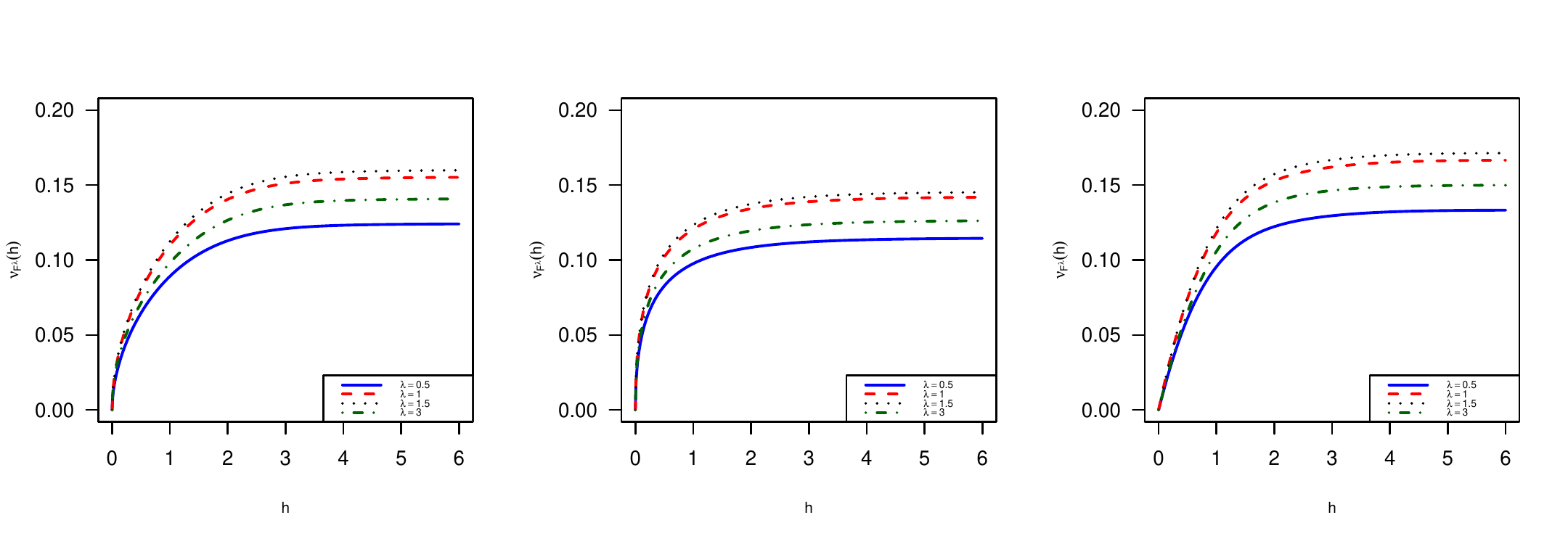}
	
	\caption{Theoretical $F^{\lambda} -$madogram functions \ref{madogram formula} with $\lambda \in \{0.5, 1, 1.5, 3\}$. \textbf{Left panel}: Max-mixture model in which in which $X$ is isotropic Brown-Resnick with variogram $2\gamma(h)= (h/0.5)^{2}$, and $Y$ is isotropic inverted extremal$-t$ process with $v=1$ degrees of freedom and $\rho(h)=\exp(- h/1.5)$. \textbf{Middle panel}: Max-mixture model in which $X$ is isotropic extremal$-t$ with $v_1=1$ degrees of freedom and $\rho(h)=\exp(- h)$, $Y$ is isotropic inverted extremal$-t$ process with $v_2=2$ degrees of freedom and $\rho(h)=\exp(- h/1.5)$.  \textbf{Right panel}: Max-mixture model in which $X$ is isotropic Brown-Resnick with variogram $2\gamma(h)= (h/0.2)^{2}$, $Y$ is isotropic inverted Brown-Resnick with variogram $2\gamma(h)= h^{2}$.}
	
	\label{TheorticalApp}
	
\end{figure}

\section*{Appendix B. } \label{sec:Appendix.B}
Supplementary Material related to this article can be found online at \href{url}{http://math.univ-lyon1.fr/homes-www/abuawwad/Supplementary/}. 

\section*{Appendix C. } \label{sec:Appendix.C}

\em Proof of Lemma \ref{lemma}.  Denoting the joint survivor function of the process $Z$ (\ref{max-mixture}) $\mathbb{P}[Z({s}_{1})>x,Z({s}_{2})>y]$ by $\bar{G}_{Z}(x,y)$, we may write\\
\begin{equation*}
	\bar{G}_Z(x,y) = 1- G_{X}(x)-G_{Y}(y)+G_{Z}(x,y)
\end{equation*}

where $G_{X}(x)$, $G_{Y}(y)$ are the marginal probability distribution functions of processes $X(.)$ and $Y(.)$ respectively in model (\ref{max-mixture}) and $G_{Z}(x,y)$ is the bivariate probability distribution function of the stochastic process $Z(.)$ in the same model. using (\ref{Biv MM}) and (\ref{distribution MM}), setting $x=y=z$
\begin{align*}
	\mathbb{P}[Z({s}_1)>z|Z({s}_2)>z]&= \frac{\mathbb{P}[Z({s}_1)>z,Z({s}_2)>z]}{\mathbb{P}[Z(s_2)>z]}\\
	&=\frac{\mathbb{P}\left(X({s}_1)> \frac{z}{\ a}, X(s_2)> \frac{z}{\ a}\right) \mathbb{P}\left(Y({s}_1)>\frac{z}{1-\ a},Y(s_2)>\frac{z}{1-\ a}\right)}{1-\mathbb{P}[Z(s_2)\leq z]}\\
	&=\frac{1- 2e^{-\frac{1}{z}}+ e^ {- a V_{X}(z,z)}\left\{-1+ 2 e ^ {-\frac{1-a}{z}}+ e^ {[-V_{Y}\left[\omega\left(\frac{1-a}{z}\right),\omega\left(\frac{1-a}{z}\right)\right] } \right\}}{1-e^{-\frac{1}{z}}}
\end{align*}
where $\omega\left(\frac{1-a}{z}\right)= -1/\log \left[1- e^{-\frac{1-a}{z}} \right]$, taking the advantage of the homogeneity of order $-1$ of the exponent measures $V_{X}(.)$ and $V_{Y}(.)$; we have $V_{X}(z,z)= \theta_{X}/z$,  $e^{[-V_{Y}\left[\omega\left(\frac{1-a}{z}\right),\omega\left(\frac{1-a}{z}\right)\right]}= \left[{1-e ^ {-\frac{1-a}{z}}}\right]^{\theta_{Y}}$ and this gives (\ref{conditional prob}).
\\


	\bibliographystyle{plainnat}
	\bibliography{Abd}
	
\end{document}